\documentclass[11pt,a4paper,twoside]{article}
\usepackage[hmarginratio=1:1,bmargin=1.3in]{geometry}
\usepackage[frenchb]{babel}
\usepackage[utf8]{inputenc}
\usepackage[OT2,T1]{fontenc}
\usepackage{amsmath, amsthm}
\usepackage{amsfonts}
\usepackage{amssymb}
\usepackage[all]{xy}
\usepackage{graphicx}
\usepackage{color}
\usepackage[nottoc, notlof, notlot]{tocbibind}
\usepackage{array}
\usepackage{url}
\usepackage{rotating}
\usepackage{fancyhdr}
\usepackage{fancyhdr}
\usepackage{titlesec}
\usepackage{xfrac} 
\usepackage{macro}
\usepackage{mathtools}
\title{\textbf{Vanishing theorems and Brauer-Hasse-Noether exact sequences for the cohomology of higher-dimensional fields}}
\author{Diego Izquierdo\\
\small Département de Mathématiques et Applications - École Normale Supérieure\\
\small CNRS, PSL Research University - 45, Rue d'Ulm - 75005 Paris - France\\
\small \url{diego.izquierdo@ens.fr}}
\date{June 2018}

\titleformat{\section}[hang]{\center\Large\bf}{\thesection.}{0.5cm}{}

\DeclareSymbolFont{cyrletters}{OT2}{wncyr}{m}{n}
\DeclareMathSymbol{\Sha}{\mathalpha}{cyrletters}{"58}
\DeclareMathSymbol{\Brusse}{\mathalpha}{cyrletters}{"42}
\addto\captionsfrenchb{}

\theoremstyle{plain}
\newtheorem{theorem}{Theorem}[section]
\newtheorem{lemma}[theorem]{Lemma}
\newtheorem{proposition}[theorem]{Proposition}
\newtheorem{corollary}[theorem]{Corollary}
\newtheorem{definition}[theorem]{Definition}
\newtheorem{thmx}{Theorem}

\theoremstyle{definition}
\newtheorem{remarque}[theorem]{Remark}

\newtheorem{notation}[theorem]{Notation}

\newtheoremstyle{hypo}  
  {\topsep}   
  {\topsep}   
  {\itshape}  
  {1.5ex}       
  {\bfseries} 
  {)}         
  {8pt plus 1pt minus 1pt}  
  {}          
\theoremstyle{hypo}

\newtheoremstyle{hypol}  
  {\topsep}   
  {\topsep}   
  {\itshape}  
  {1.5ex}       
  {\bfseries} 
  {)$_{\ell}$}         
  {8pt plus 1pt minus 1pt}  
  {}          
\theoremstyle{hypol}

\newtheoremstyle{DP}  
  {\topsep}   
  {\topsep}   
  {\itshape}  
  {}       
  {\bfseries} 
  {.}         
  {8pt plus 1pt minus 1pt}  
  {}          
\theoremstyle{DP}

  \def\commutatif{\ar@{}[rd]|{(\star)}}

\DeclareMathSymbol{\ch}{\mathalpha}{cyrletters}{"51}
\begin{document}

\maketitle

\small

\textbf{Abstract.} Let $k$ be a finite field, a $p$-adic field or a number field. Let $K$ be a finite extension of the Laurent series field in $m$ variables $k((x_1,...,x_m))$. When $r$ is an integer and $\ell$ is a prime number, we consider the Galois module $\mathbb{Q}_{\ell}/\mathbb{Z}_{\ell}(r)$ over $K$ and we prove several vanishing theorems for its cohomology. In the particular case when $K$ is a finite extension of the Laurent series field in two variables $k((x_1,x_2))$, we also prove exact sequences that play the role of the Brauer-Hasse-Noether exact sequence for the field $K$ and that involve some of the cohomology groups of $\mathbb{Q}_{\ell}/\mathbb{Z}_{\ell}(r)$ which do not vanish.\\

\textbf{MSC classes:} 11R34, 11G25, 11G35, 11S25, 12G05, 14G15, 14G20, 14G25, 14G27, 14J20, 14B05, 14J17. \\

\textbf{Keywords:} Galois cohomology, Bloch-Kato conjecture, Laurent series fields in two or more variables, function fields in two or more variables, singularities, finite base fields, $p$-adic base fields, global base fields, Hasse principle, Brauer group, Brauer-Hasse-Noether exact sequence.

\normalsize

\section{Introduction}

\hspace{3ex} Let $K$ be a field and let $\ell$ be a prime number different from the characteristic of $K$. For each integer $r$, one can define the Galois module $\mathbb{Q}_{\ell}/\mathbb{Z}_{\ell}(r)$ as:
$$\mathbb{Q}_{\ell}/\mathbb{Z}_{\ell}(r):= \varinjlim_n \mathbb{Z}/\ell^n\mathbb{Z}(r),$$
where:
$$\mathbb{Z}/\ell^n\mathbb{Z}(r):= \begin{cases} \mu_{\ell^n}^{\otimes r}, & \mbox{ if } r\geq 0 \\ \text{Hom}(\mu_{\ell^n}^{\otimes (-r)},\mathbb{Z}/\ell^n\mathbb{Z}), & \mbox{otherwise. }  \end{cases}$$
 The Galois modules $\mathbb{Q}_{\ell}/\mathbb{Z}_{\ell}(r)$ and their cohomology play an important role when one wants to study the arithmetic of the field $K$.\\
 
 \hspace{3ex} The starting point of the present paper is the article \cite{jannsen2}, in which Jannsen deals with the case when $K$ is the function field $k(X)$ of a curve $X$ defined over a $p$-adic field or a number field $k$. When $k$ is $p$-adic, he proves that the group $H^3(K,\mathbb{Q}_{\ell}/\mathbb{Z}_{\ell}(r))$ vanishes provided that $r\neq 2$, and when $k$ is a number field, he proves that:
$$ H^3(K,\mathbb{Q}_{\ell}/\mathbb{Z}_{\ell}(r)) \cong \bigoplus_{\pi \in \Omega_k^{\infty}} H^3(K \cdot k_{\pi}^h,\mathbb{Q}_{\ell}/\mathbb{Z}_{\ell}(r))$$
 for $r\neq 2$ and that there is an exact sequence:
\begin{equation} \label{sejan}
 0 \rightarrow H^3(K,\mathbb{Q}_{\ell}/\mathbb{Z}_{\ell}(2)) \rightarrow \bigoplus_{\pi \in \Omega_k} H^3(K \cdot k_{\pi}^h,\mathbb{Q}_{\ell}/\mathbb{Z}_{\ell}(2)) \rightarrow \bigoplus_{ v\in X^{(1)}} \mathbb{Q}_{\ell}/\mathbb{Z}_{\ell} \rightarrow \mathbb{Q}_{\ell}/\mathbb{Z}_{\ell} \rightarrow 0.
 \end{equation}
 Here, $\Omega_k$ (resp. $\Omega_k^{\infty}$) stands for the set of places of $k$ (resp. infinite places of $k$), $k_{\pi}^h$ denotes the henselization of $k$ at $\pi$ for $\pi\in \Omega_k$, and $X^{(1)}$ is the set of codimension 1 points of $X$. 
 \\
 
 \hspace{3ex} Since Jannsen's work, the cohomology of the Galois module $\mathbb{Q}_{\ell}/\mathbb{Z}_{\ell}(r)$ over fields of arithmetico-geometrical nature has been studied by several authors. Indeed, when $K$ is the function field of a smooth variety of any dimension defined over a finite field, a $p$-adic field or a number field, Kahn, Pirutka, Saito and Sato have proved several vanishing theorems for the cohomology of $\mathbb{Q}_{\ell}/\mathbb{Z}_{\ell}(r)$ (see \cite{kahn}, \cite{pirutka}, \cite{saitosato}). Moreover, when $K$ is the function field of an $n$-dimensional smooth variety defined over a number field, Jannsen has also generalized the injectivity of the first map in exact sequence (\ref{sejan}) by proving that there is a local-glocal principle for the group $H^{n+2}\left(K,\mathbb{Q}_{\ell}/\mathbb{Z}_{\ell}(n+1)\right)$ (see \cite{jannsen}). Note that all these results have found various arithmetical applications: for instance, they have been used to study some properties of quadratic forms (\cite{jannsen}), of Bloch-Ogus complexes (\cite{kahn}), of Chow groups (\cite{saitosato}) and of some birational invariants (\cite{pirutka}).  \\
 
 \hspace{3ex}  Given a field $K$ and a prime number $\ell$ different from the characteristic of $K$, the main goal of the present article consists in studying the groups $H^d(K,\mathbb{Q}_{\ell}/\mathbb{Z}_{\ell}(r))$ when $d$ is the cohomological dimension of the field $K$. The main case we are interested in is the case when $K$ is a finite extension of a Laurent series field in any number of variables with coefficients in a finite field, a $p$-adic field or a totally imaginary number field. Note that such fields naturally arise as completions of varieties at closed points. Geometrically speaking, this means that, contrary to the cases previously studied by Jannsen, Kahn, Pirutka, Saito and Sato, we are here interested in a field $K$ that arises as the function field of a \textit{singular} scheme. \\

\hspace{3ex} It turns out that our results also apply when $K$ is a finite extension of a field of the form $k((x_1,...,x_n))(y_1,...,y_m)$ with $k$ a finite field, a $p$-adic field or a totally imaginary number field. They therefore vastly generalize the previous results of Jannsen, Kahn, Pirutka, Saito and Sato for finite extensions of $k(y_1,...,y_m)$. The methods we use are radically different from theirs and they provide a unified framework to study the Galois cohomology of the modules $\mathbb{Q}_{\ell}/\mathbb{Z}_{\ell}(r)$.

\subsection*{Organization of the article and main results}

\hspace{3ex} One of the main ingredients in this article is the Bloch-Kato conjecture, which states that, for each field $k$, each prime $\ell$ different from the characteristic of $k$ and each positive integer $m$, the morphism induced by the cup-product:
$$H^1(k,\mathbb{Z}/\ell^m\mathbb{Z}(1))^{\otimes n} \rightarrow H^n(k,\mathbb{Z}/\ell^m\mathbb{Z}(n))$$
is surjective. Note that the Bloch-Kato conjecture deals with each power of $\ell$ separately, while we need to work with all the powers of $\ell$ simultaneously since we are interested in the cohomology of $\mathbb{Q}_{\ell}/\mathbb{Z}_{\ell}(r)$. The main goal of section \ref{secbk} consists in revisiting the Bloch-Kato conjecture in order to prove a statement taking into account all powers of $\ell$ at the same time (theorem \ref{surj}). This statement is of independent interest.\\

\hspace{3ex} Now let $n$ and $d$ be non-negative integers and let $k$ be a field of cohomological dimension $d$. In section \ref{sec1}, we are interested in the cohomology of the Galois module $\mathbb{Q}_{\ell}/\mathbb{Z}_{\ell}(r)$ when $K$ is a finite extension of the Laurent series field in $n$ variables over $k$. The main theorem is an abstract result that, given an integer $r$, states the vanishing of the group $H^{n+d}(K,\mathbb{Q}_{\ell}/\mathbb{Z}_{\ell}(r))$ provided that some conditions on the Galois cohomology of the base field $k$ are fulfilled (theorem \ref{abs1}). This theorem applies whatever the field $k$ is and one of the main ingredients of its proof is the variant of the Bloch-Kato conjecture proved in section \ref{secbk}.\\

\hspace{3ex} In section \ref{app1}, we apply the previous abstract theorem to the specific situations in which $k$ is a finite field, a $p$-adic field or a number field. For example, we prove the following theorem:

\begin{thmx}\label{A'} \emph{(Theorems \ref{thpadique1} and \ref{thcdn1})}\\
Let $\ell$ be a prime number and let $n$ be a non-negative integer. Let $K$ be a finite extension of the Laurent series field $k((x_1,...,x_n))$ over a base field $k$. If $k$ is a $p$-adic field or a totally imaginary number field or if $k$ is any number field and $\ell \neq 2$, then $H^{n+2}(K,\mathbb{Q}_{\ell}/\mathbb{Z}_{\ell}(r))$ vanishes for all $r\not\in \left[ 1,n+1 \right] $. 

\end{thmx}

The proof is based on the abstract theorem \ref{abs1} and it involves the Weil conjectures, 1-motives, $p$-adic Hodge theory and a local-global principle due to Jannsen. The results of section \ref{sec1} are actually far more general than theorem \ref{A'}, since they also apply when $k$ is finite or when $K$ is a finite extension of $k((x_1,...,x_n))(y_1,...,y_m)$. In this sense, they generalize and unify previous results for finite extensions of $k(y_1,...,y_m)$. Moreover, our results still hold when one replaces the Laurent series field $k((x_1,...,x_n))$ by the fraction field of any henselian, normal, excellent, local ring with finite residue field. Such a local ring may have mixed characteristic. \\ 

\hspace{3ex} Section \ref{(())} is devoted to the refinement of the results of section \ref{sec1} in the particular case when $K$ is a finite extension of the Laurent series field in two variables $k((x,y))$ over some base field $k$ of cohomological dimension $d$. The main theorem of paragraph \ref{ppp} is an abstract result that, given an integer $r$, states the vanishing of the group $H^{d+2}(K,\mathbb{Q}_{\ell}/\mathbb{Z}_{\ell}(r))$ provided that some conditions on the Galois cohomology of the base field $k$ are fulfilled (theorem \ref{abs}). Its proof is based on a careful study of the Brauer-Hasse-Noether exact sequence for fields like $\mathbb{C}((x,y))$ that has been proved in \cite{diego2}. \\

\hspace{3ex} In paragraph \ref{app2}, we apply the previous abstract theorem to the specific situations in which $k$ is a finite field, a $p$-adic field or a totally imaginary number field. In particular, we prove the following theorem, which improves theorem \ref{A'} in the 2-dimensional case:

\begin{thmx}\label{B} \emph{(Theorem \ref{thdim2})}\\
Let $k$ be a field and let $\ell$ be a prime number. Let $K$ be a finite extension of the Laurent series field in two variables over $k$.
\begin{itemize}
\item[(i)] If $k$ is finite and if $\ell$ is different from the characteristic of $k$, then $H^3(K,\mathbb{Q}_{\ell}/\mathbb{Z}_{\ell}(r))$ vanishes for all $r\neq 2$.
\item[(ii)] If $k$ is a $p$-adic field or a totally imaginary number field or if $k$ is any number field and $\ell \neq 2$, then $H^4(K,\mathbb{Q}_{\ell}/\mathbb{Z}_{\ell}(r))$ vanishes for all $r\neq 3$.
\end{itemize}
\end{thmx}

The results of section \ref{app2} are in fact more general than theorem \ref{B}, since they apply to the fraction field of a geometrically integral, normal, henselian, excellent, 2-dimensional $k$-algebra with residue field $k$ whenever $k$ is a finite field, a $p$-adic field or a totally imaginary number field.\\

\hspace{3ex} Finally, paragraph \ref{parbk} is devoted to the study of those cohomology groups of $\mathbb{Q}_{\ell}/\mathbb{Z}_{\ell}(r)$ which do not always vanish in the case when $K$ is a finite extension of the Laurent series field in two variables $k((x,y))$ over some base field $k$. The main result is corollary \ref{corobhn}: it settles several exact sequences which involve those groups and which should be understood as Brauer-Hasse-Noether exact sequences for the field $K$. In the case when $k$ is finite, this allows us to recover a result of Saito (theorem 5.2 of \cite{saito}). When $k$ is a $p$-adic field or a number field, we obtain several new results:

\begin{thmx}\label{C} \emph{(Corollaries \ref{padicbhn} and \ref{bhncdn} and \ref{bhncdnbis})}\\
Let $k$ be a field and let $\ell$ be a prime number. Let $K$ be a finite extension of the Laurent series field in two variables $k((x,y))$ over $k$ and assume that $k$ is algebraically closed in $K$. Denote by $X$ the spectrum of the integral closure of the formal power series ring $k[[x,y]]$ in $K$.
\begin{itemize}
\item[(i)] Assume that $k$ is a $p$-adic field. Then we have an exact sequence:
$$0\rightarrow (\mathbb{Q}_{\ell}/\mathbb{Z}_{\ell})^{\rho} \rightarrow H^{4}(K,\mathbb{Q}_{\ell}/\mathbb{Z}_{\ell}(3)) \rightarrow \bigoplus_{v\in X^{(1)}} H^{4}(K_v,\mathbb{Q}_{\ell}/\mathbb{Z}_{\ell}(3)) \rightarrow (\text{Br}\; k) \{\ell\} \rightarrow 0$$
for some $\rho\geq 0$ which can be bounded by some geometrical invariants associated to $K$.
\item[(ii)] Assume that $k$ is a totally imaginary number field or that $k$ is any number field and $\ell \neq 2$. We have an exact sequence: 
$$0 \rightarrow D \rightarrow H^{4}(K,\mathbb{Q}_{\ell}/\mathbb{Z}_{\ell}(3)) \rightarrow \bigoplus_{v\in X^{(1)}} H^{4}(K_v,\mathbb{Q}_{\ell}/\mathbb{Z}_{\ell}(3)) \rightarrow (\text{Br}\; k) \{\ell\} \rightarrow 0$$
for some divisible group $D$. 
\item[(iii)] Assume that $k$ is any number field. For $\pi \in \Omega_k$, denote by $k_{\pi}^h$ the henselization of $k$ at $\pi$ and by $K_{\pi}^h$ the field $K\otimes_k k_{\pi}^h$. Then there exists a natural complex $\mathcal{C}_K$:
\small
\begin{equation}
\xymatrix@R=2mm@C=5mm{
& \text{(degree 1)} & \text{(degree 2)} & \text{(degree 3)} & \text{(degree 4)} & \\
0 \ar[r] &  H^4(K,\mathbb{Q}_{\ell}/\mathbb{Z}_{\ell}(3)) \ar[r] & \bigoplus\limits_{\pi\in \Omega_k} H^4(K_{\pi}^h,\mathbb{Q}_{\ell}/\mathbb{Z}_{\ell}(3)) \ar[r] & \bigoplus\limits_{v\in X^{(1)}} \mathbb{Q}_{\ell}/\mathbb{Z}_{\ell} \ar[r] & \mathbb{Q}_{\ell}/\mathbb{Z}_{\ell} \ar[r] & 0
}
\end{equation}
\normalsize
whose homology satisfies the following properties:
\begin{enumerate}
\item[(a)] the groups $H^3(\mathcal{C}_K)$ and $ H^4(\mathcal{C}_K)$ are trivial;
\item[(b)] the groups $H^1(\mathcal{C}_K)$ and $H^2(\mathcal{C}_K)$ can be precisely controlled thanks to some geometrical invariants related to the combinatorics of the singularities attached to the field $K$;
\item[(c)] the group $H^2(\mathcal{C}_K)$ is isomorphic to $(\mathbb{Q}_{\ell}/\mathbb{Z}_{\ell})^{\rho}$ for some $\rho \geq 0$;
\item[(d)] the group $H^1(\mathcal{C}_K)$ has finite exponent.
\end{enumerate}
\end{itemize}
\end{thmx}

Part (iii) of the previous theorem should be seen as a generalization of Jannsen's exact sequence (\ref{sejan}) to the field $K$.

\subsection*{Some general notations}

\begin{itemize}
\item[$\bullet$] When $k$ is a field, $k^s$ stands for a separable closure of $k$. When $k'/k$ is a finite field extension and $M$ is a Galois module over $k'$, the notation $I_{k'/k}(M)$ stands for the corresponding induced module.
\item[$\bullet$] When $A$ is an abelian group, $n$ is a positive integer and $\ell$ is a prime number, ${_n}A$ stands for the $n$-torsion subgroup of $A$, $A\{\ell\}$ is the $\ell$-primary torsion subgroup of $A$ and $T_{\ell}A$ is the projective limit $\varprojlim_n {_{\ell^n}}A$.
\item[$\bullet$] When $A$ is a topological group, $A^D$ stands for the group of continuous homomorphisms from $A$ to $\mathbb{Q}/\mathbb{Z}$.
\item[$\bullet$] A topological $\mathbb{Z}_{\ell}$-module $A$ is said to be pseudo compact if it is Hausdorff and complete and it has a basis of neighborhoods of 0 comprised of submodules $A'$ of $A$ such that $A/A'$ has finite length. When $A$ and $B$ are pseudocompact $\mathbb{Z}_{\ell}$-modules, $A\hat{\otimes}_{\mathbb{Z}_{\ell}} B$ denotes the completed tensor product of $A$ and $B$. For more details on this notion, see chapter $\mathrm{VII_B}$ of \cite{sga3}.
\end{itemize}

\section{Passing to the limit in the Bloch-Kato conjecture}\label{secbk}

\subsection{A modified tensor product for torsion groups}\label{par1}

\hspace{3ex} This section is devoted to the definition of a modified tensor product for torsion groups.  

\begin{definition}
Let $\ell$ be a prime number and $r$ a positive integer. Let $A$ and $B$ be two $\ell$-primary torsion abelian groups. Define:
\begin{gather*}
A \boxtimes B := \left( A^D \hat{\otimes}_{\mathbb{Z}_{\ell}} B^D \right)^D,\\
A^{\boxtimes r} := A \boxtimes ... \boxtimes A \;\;\;\; (r \; \mathrm{times).}
\end{gather*}
By convention, we let $A^{\boxtimes 0} :=\mathbb{Q}_{\ell}/\mathbb{Z}_{\ell}$.
\end{definition}

Observe that, in the previous definition, $A$ and $B$ are endowed with the discrete topology. The groups $A^D$, $B^D$ and $A^D \hat{\otimes}_{\mathbb{Z}_{\ell}} B^D$ are therefore profinite, and the group $A \boxtimes B$ is discrete.\\

\hspace{3ex} The following lemma summarizes some formal properties of the operation $\boxtimes$:

\begin{lemma}\label{formal}
Let $\ell$ be a prime number.
\begin{itemize}
\item[(i)] Let $A$ be an $\ell$-primary torsion abelian group. There is a canonical isomorphism:
\begin{equation}\label{w1}
A \boxtimes \mathbb{Q}_{\ell}/\mathbb{Z}_{\ell} \cong A.
\end{equation}
\item[(ii)] Let $(A_i)$ and $(B_j)$ be filtrant direct systems of $\ell$-primary torsion abelian groups. There is a canonical isomorphism:
\begin{equation}\label{w2}
(\varinjlim_i A_i) \boxtimes (\varinjlim_j B_j) \cong \varinjlim_{i,j} (A_i \boxtimes B_j).
\end{equation}
\item[(iii)] Let $A$, $B$ and $C$ be three $\ell$-primary torsion abelian groups. There is a canonical isomorphism:
\begin{equation}\label{w3}
(A \boxtimes B)  \boxtimes C  \cong \left( A^D \hat{\otimes}_{\mathbb{Z}_{\ell}} B^D \hat{\otimes}_{\mathbb{Z}_{\ell}} C^D \right)^D .
\end{equation}
In particular:
$$(A \boxtimes B)  \boxtimes C \cong A \boxtimes (B  \boxtimes C).$$
\item[(iv)] Let $\mathcal{A}_{\ell}$ be the category of $\ell$-primary torsion abelian groups, and let $A$ be an object of $\mathcal{A}_{\ell}$. The category $\mathcal{A}_{\ell}$ has enough injectives and the functor $F_A := - \boxtimes A$ is a covariant left-exact functor from $\mathcal{A}_{\ell}$ into itself. Hence one may define the derived functors $\text{Tor}_{\ell}^i(-,A):=R^iF_A$.
\end{itemize}

\begin{remarque}~
\begin{itemize}
\item[(i)] In the sequel, we will use the notation $A \boxtimes B \boxtimes C$.
\item[(ii)] If $A$ is an $\ell$-primary torsion abelian group, the functor $- \boxtimes A$ may not be right-exact. For instance, multiplication by $\ell$ on $\mathbb{Q}_{\ell}/\mathbb{Z}_{\ell}$ is surjective but multiplication by $\ell$ on $\mathbb{Z}/ \ell\mathbb{Z} \cong \mathbb{Q}_{\ell}/\mathbb{Z}_{\ell}\boxtimes  \mathbb{Z}/ \ell\mathbb{Z}$ is not.
\item[(iii)] Similarly to the case of the usual tensor product, one can prove the following properties:
\begin{itemize}
\item[a)] if $A$ and $B$ are two $\ell$-primary torsion abelian groups, then $\text{Tor}_{\ell}^i(A,B)=0$ for $i\geq 2$;
\item[b)] if $A$ and $B$ are two $\ell$-primary torsion abelian groups, then there is a natural isomorphism $\text{Tor}_{\ell}^1(A,B)\cong \text{Tor}_{\ell}^1(B,A)$;
\item[c)] $\text{Tor}^1_{\ell}(\mathbb{Z}/\ell^n\mathbb{Z},\mathbb{Z}/\ell^m\mathbb{Z})=\mathbb{Z}/\ell^{\min \{m,n\}}\mathbb{Z}$.
\end{itemize}
\end{itemize}
\end{remarque}

\begin{proof}[Proof]
\begin{itemize}
\item[(i)] We have canonical isomorphisms:
$$A \boxtimes \mathbb{Q}_{\ell}/\mathbb{Z}_{\ell} \cong \left( A^D \hat{\otimes}_{\mathbb{Z}_{\ell}} \mathbb{Z}_{\ell} \right)^D \cong (A^D)^D \cong A.$$
\item[(ii)] Since $\hat{\otimes}_{\mathbb{Z}_{\ell}}$ commutes with inverse limits, we have canonical isomorphisms:
\begin{align*}
(\varinjlim_i A_i) \boxtimes (\varinjlim_j B_j) & \cong \left( \varprojlim_i (A_i^D) \hat{\otimes}_{\mathbb{Z}_{\ell}} \varprojlim_j (B_j^D) \right)^D\\
& \cong \left( \varprojlim_{i,j} (A_i^D \hat{\otimes}_{\mathbb{Z}_{\ell}} B_j^D) \right)^D\\
& \cong \varinjlim_{i,j} (A_i \boxtimes B_j).
\end{align*}
\item[(iii)] We have canonical isomorphisms:
$$(A \boxtimes B)  \boxtimes C  \cong  \left( \left(\left( A^D \hat{\otimes}_{\mathbb{Z}_{\ell}} B^D \right)^D\right)^D \hat{\otimes}_{\mathbb{Z}_{\ell}} C^D \right)^D \cong \left( A^D \hat{\otimes}_{\mathbb{Z}_{\ell}} B^D \hat{\otimes}_{\mathbb{Z}_{\ell}} C^D \right)^D.$$

\item[(iv)] The category $\mathcal{A}_{\ell}$ has enough injectives since every $\ell$-primary torsion abelian group embeds in a divisible $\ell$-primary torsion abelian group. Let's now prove that the functor $F_A$ is left-exact. To do so, consider an exact sequence of $\ell$-primary torsion abelian groups:
$$0\rightarrow B_1\rightarrow B_2 \rightarrow B_3 \rightarrow 0.$$
Since $B_1$, $B_2$ and $B_3$ are discrete torsion groups, we get a dual exact sequence:
$$0\rightarrow B_3^D\rightarrow B_2^D \rightarrow B_1^D \rightarrow 0.$$
The functor $- \hat{\otimes}_{\mathbb{Z}_{\ell}} A^D$ is right-exact, so that we obtain an exact sequence:
$$B_3^D \hat{\otimes}_{\mathbb{Z}_{\ell}} A^D\rightarrow B_2^D \hat{\otimes}_{\mathbb{Z}_{\ell}} A^D \rightarrow B_1^D \hat{\otimes}_{\mathbb{Z}_{\ell}} A^D \rightarrow 0.$$
Since all the groups in the previous exact sequence are profinite, we get a dual exact sequence:
$$0\rightarrow B_1\boxtimes A \rightarrow B_2\boxtimes A \rightarrow B_3 \boxtimes A.$$
\end{itemize}
\end{proof}

\end{lemma}

In the sequel, we will often consider the group $A\boxtimes B$ when $A$ and $B$ are $\ell$-primary torsion Galois modules over a field $k$. In that case, note that the $\ell$-primary torsion group  $A\boxtimes B$ is naturally endowed with the structure of a Galois module over $k$. Moreover, for each integer $i$, isomorphism (\ref{w1}) induces an isomorphism of Galois modules:
$$A \boxtimes \mathbb{Q}_{\ell}/\mathbb{Z}_{\ell}(i) \cong A(i),$$
and isomorphisms (\ref{w2}) and (\ref{w3}) are isomorphisms of Galois modules.\\

\hspace{3ex} We finish this paragraph with two technical lemmas which will be useful in section \ref{sec1}:

\begin{lemma}\label{abstractdiv}
Let $\ell$ be a prime number and let $k$ be a field with characteristic different from $\ell$. Let $A,B,C,D$ be four $\ell$-primary torsion Galois modules over $k$. Let $N$ be a positive integer. We make the following assumptions:
\begin{itemize}
\item[(1)] the Galois module $A$ is divisible;
\item[(2)] we have an exact sequence of Galois modules:
\begin{equation}\label{abstractnonsense}
0 \rightarrow A \rightarrow B \rightarrow C \rightarrow 0;
\end{equation}
\item[(3)] for each $j\in \{0,...,N\}$, the abelian group $H^d(k,A^{\boxtimes j} \boxtimes C^{\boxtimes N-j}\boxtimes D)$ has finite exponent.
\end{itemize} 
Then the group $H^d(k,A^{\boxtimes j_1} \boxtimes B^{\boxtimes j_2}\boxtimes C^{\boxtimes j_3}\boxtimes D)$ has finite exponent for all triples $(j_1,j_2,j_3)$ such that $j_1+j_2+j_3=N$.
\end{lemma}

\begin{proof}[Proof]
We proceed by induction on $j_2$:
\begin{itemize}
\item[$\bullet$] The case $j_2=0$ is assumption (3).
\item[$\bullet$] Assume that we have proved the lemma for some $j_2$. Let $j_1$ and $j_3$ be non-negative integers such that $j_1+(j_2+1)+j_3=N$. Since $A$ is divisible, exact sequence (\ref{abstractnonsense}) is split as a sequence of abelian groups. But $\boxtimes$ is compatible with direct sums, so we can apply the functor $- \boxtimes \left(A^{\boxtimes j_1} \boxtimes B^{\boxtimes j_2} \boxtimes C^{\boxtimes j_3} \boxtimes D\right)$ to exact sequence (\ref{abstractnonsense}) and we get the following exact sequence of Galois modules:
\begin{equation} \label{suitform}
0\rightarrow A^{\boxtimes j_1+1} \boxtimes B^{\boxtimes j_2} \boxtimes C^{\boxtimes j_3} \boxtimes D \rightarrow A^{\boxtimes j_1} \boxtimes B^{\boxtimes j_2+1} \boxtimes C^{\boxtimes j_3} \boxtimes D \rightarrow A^{\boxtimes j_1} \boxtimes B^{\boxtimes j_2} \boxtimes C^{\boxtimes j_3+1} \boxtimes D \rightarrow 0.
\end{equation}
By assumption, the groups $$H^d(k,A^{\boxtimes j_1+1} \boxtimes B^{\boxtimes j_2} \boxtimes C^{\boxtimes j_3}\boxtimes D) \;\;\; \mathrm{and} \;\;\; H^d(k,A^{\boxtimes j_1} \boxtimes B^{\boxtimes j_2} \boxtimes C^{\boxtimes j_3+1}\boxtimes D)$$ have finite exponent. Hence it follows from exact sequence (\ref{suitform}) that the abelian group $H^d(k,A^{\boxtimes j_1} \boxtimes B^{\boxtimes j_2+1}\boxtimes C^{\boxtimes j_3}\boxtimes D)$ has finite exponent.
 \end{itemize}
\end{proof}

\begin{lemma}\label{abstractfini}
Let $\ell$ be a prime number and let $k$ be a field with characteristic different from $\ell$. Let $A,B,C,D$ be four $\ell$-primary torsion Galois modules over $k$. Let $N$ be a positive integer. We make the following assumptions:
\begin{itemize}
\item[(1)] the Galois module $A$ has finite exponent;
\item[(2)] we have an exact sequence of Galois modules:
\begin{equation}\label{abstractnonsense2}
0 \rightarrow A \rightarrow B \rightarrow C \rightarrow 0;
\end{equation}
\item[(3)] the abelian group $H^d(k,C^{\boxtimes N}\boxtimes D)$ has finite exponent.
\end{itemize} 
Then the group $H^d(k, B^{\boxtimes j_2}\boxtimes C^{\boxtimes j_3}\boxtimes D)$ has finite exponent for all pairs $(j_2,j_3)$ such that $j_2+j_3=N$.
\end{lemma}

\begin{proof}[Proof]
We proceed by induction on $j_2$:
\begin{itemize}
\item[$\bullet$] The case $j_2=0$ is assumption (3).
\item[$\bullet$] Assume that we have proved the lemma for some $j_2$. Let $j_3$ be a non-negative integer such that $(j_2+1)+j_3=N$. Since the functor $- \boxtimes \left(B^{\boxtimes j_2} \boxtimes C^{\boxtimes j_3} \boxtimes D\right)$ is left-exact, we get an exact sequence:
\begin{equation*} 
0\rightarrow A \boxtimes B^{\boxtimes j_2} \boxtimes C^{\boxtimes j_3} \boxtimes D \rightarrow  B^{\boxtimes j_2+1} \boxtimes C^{\boxtimes j_3} \boxtimes D \rightarrow  B^{\boxtimes j_2} \boxtimes C^{\boxtimes j_3+1} \boxtimes D  \rightarrow \text{Tor}_{\ell}^1(A,B^{\boxtimes j_2} \boxtimes C^{\boxtimes j_3} \boxtimes D).
\end{equation*}
Since the group $\text{Tor}_{\ell}^1(A,B^{\boxtimes j_2} \boxtimes C^{\boxtimes j_3} \boxtimes D)$ has finite exponent, we deduce that there are two Galois modules $I$ and $J$ and two exact sequences of Galois modules:
\begin{gather} \label{es1}
0\rightarrow A \boxtimes B^{\boxtimes j_2} \boxtimes C^{\boxtimes j_3} \boxtimes D \rightarrow  B^{\boxtimes j_2+1} \boxtimes C^{\boxtimes j_3} \boxtimes D \rightarrow  I \rightarrow 0,\\
0\rightarrow I \rightarrow B^{\boxtimes j_2} \boxtimes C^{\boxtimes j_3+1} \boxtimes D \rightarrow J \rightarrow 0,\label{es2}
\end{gather}
such that $J$ has finite exponent.
  But, by assumption, the group $H^d(k,B^{\boxtimes j_2} \boxtimes C^{\boxtimes j_3+1}\boxtimes D)$ has finite exponent. Hence it follows from exact sequences (\ref{es1}) and (\ref{es2}) that the group $H^d(k, B^{\boxtimes j_2+1}\boxtimes C^{\boxtimes j_3}\boxtimes D)$ has finite exponent.
 \end{itemize}
\end{proof}

\subsection{A consequence of the Bloch-Kato conjecture}\label{BK}

\hspace{3ex} When $\ell$ is a prime number and $K$ is a field of characteristic different from $\ell$, the Bloch-Kato conjecture states that the morphism induced by the cup-product:
$$H^1(K,\mathbb{Z}/\ell^r\mathbb{Z}(1))^{\otimes N} \rightarrow H^N(K,\mathbb{Z}/\ell^r\mathbb{Z}(N))$$
is surjective for any positive integers $r$ and $N$. In this section, we prove a statement that allows one to deal with all possible values of $r$ at the same time. For this purpose, we will need to use the operation $\boxtimes$ that has been introduced in the previous section.

\begin{lemma}\label{tors}
Let $\ell$ be a prime number, $N$ a positive integer and $A$ an $\ell$-primary torsion abelian group. For each $s\geq 1$, there is a natural isomorphism:
$${_{\ell^s}}(A^{\boxtimes N}) \cong ({_{\ell^s}}A)^{\boxtimes N}.$$
Moreover, if $t\geq s$, the injection ${_{\ell^s}}(A^{\boxtimes N}) \hookrightarrow {_{\ell^t}}(A^{\boxtimes N})$ is identified to the morphism $({_{\ell^s}}A)^{\boxtimes N} \rightarrow ({_{\ell^t}}A)^{\boxtimes N}$ induced by the injection ${_{\ell^s}}A \hookrightarrow{_{\ell^t}} A$.
\end{lemma}

\begin{proof}[Proof]
There are natural isomorphisms:
\begin{align*}
{_{\ell^s}}\left( A^{\boxtimes N}  \right) & \cong {_{\ell^s}}\left[ \left(  ( A^D ) ^{\hat{\otimes}_{\mathbb{Z}_{\ell}} N}  \right)^D \right] \cong \left[ \left(  (A^D ) ^{\hat{\otimes}_{\mathbb{Z}_{\ell}} N}  \right)/\ell^s \right]^D  \\
&\cong \left[ \left(  \left(A^D \right)/\ell^s\right) ^{\hat{\otimes}_{\mathbb{Z}_{\ell}} N} \right]^D  \cong \left[ \left(  \left({_{\ell^s}}A\right)^D\right) ^{\hat{\otimes}_{\mathbb{Z}_{\ell}} N} \right]^D \cong ({_{\ell^s}}A)^{\boxtimes N}.
\end{align*}
The second part of the statement can be easily checked by following the injection  ${_{\ell^s}}(A^{\boxtimes N}) \hookrightarrow {_{\ell^t}}(A^{\boxtimes N})$ through the previous isomorphisms.
\end{proof}

\begin{lemma}\label{torsbis}
Let $\ell$ be a prime number and $N$ a positive integer. Let $A$ be an $\ell$-primary torsion abelian group and $A_s$ the $\ell^s$-torsion subgroup of $A$ for each $s \geq 1$. Denote by $i_{s,t}$ the morphism $(A_s)^{\boxtimes N} \rightarrow (A_t)^{\boxtimes N}$ induced by the injection $A_s \hookrightarrow A_t$ when $t\geq s$. Assume that $A$ is divisible.
\begin{itemize}
\item[(i)] Let $t \geq s \geq 1$ be integers. One can define a morphism:
$$f_{s,t}:A_s^{\otimes N} \rightarrow A_t^{\otimes N}$$
in the following way: for $a_1,...,a_N \in A_s$, one sets: $$f_{s,t}(a_1\otimes ... \otimes a_N)=\ell^{t-s}(b_1\otimes ... \otimes b_N)$$ where each $b_i \in A_t$ satisfies the equality $\ell^{t-s}b_i=a_i$.
\item[(ii)] There are natural isomorphisms:
$$\varphi_s: (A_s)^{\otimes N} \rightarrow (A_s)^{\boxtimes N}$$
such that the following diagram commutes:
\begin{equation}\label{commut}
\xymatrix{
(A_s)^{\otimes N} \ar[r]^{\varphi_s}\ar[d]^{f_{s,t}} & (A_s)^{\boxtimes N} \ar[d]^{i_{s,t}}\\
(A_t)^{\otimes N} \ar[r]^{\varphi_t} & (A_t)^{\boxtimes N}
}
\end{equation}
for all $t\geq s \geq 1$.
\end{itemize}

\end{lemma}

\begin{proof}[Proof]
\item[(i)] Note first that, given $a_1,...,a_N \in A_s$, one can always find $b_1,...,b_N \in A_t$ such that $\ell^{t-s}b_i=a_i$ for $i\in \{1,...,N\}$ since $A$ is divisible. Moreover, if $b_1',...,b_N'$ are other elements of $A_t$ such that $\ell^{t-s}b_i'=a_i$ for $i\in \{1,...,N\}$, then:
\begin{align*}
\ell^{t-s}(b_1\otimes ... \otimes b_N) &= (\ell^{t-s}b_1)\otimes b_2 \otimes ... \otimes b_N \\& = (\ell^{t-s}b_1')\otimes b_2 \otimes  ... \otimes b_N\\ &= \ell^{t-s}(b_1'\otimes b_2 \otimes  ... \otimes b_N),
\end{align*}
and by repeating this argument, a simple induction shows that $\ell^{t-s}(b_1\otimes ... \otimes b_N) = \ell^{t-s}(b_1'\otimes ... \otimes b_N')$. This proves that $f_{s,t}$ is well-defined.
\item[(ii)] By the universal property of $\hat{\otimes}_{\mathbb{Z}_{\ell}}$, for each $s>0$, there are natural isomorphisms:
\begin{equation}\label{isom}
 (A_s)^{\boxtimes N} \cong \left[ \left(  \left(A_s\right)^D\right) ^{\hat{\otimes}_{\mathbb{Z}_{\ell}} N} \right]^D \cong  \text{Mult}_c((A_s)^D \times ... \times (A_s)^D , \mathbb{Z}/\ell^s\mathbb{Z}),
 \end{equation}
where $\text{Mult}_c$ stands for the set of continuous multilinear maps. Consider the morphism:
$$\varphi_s': (A_s)^{\otimes N} \rightarrow \text{Mult}_c(A_s^D \times ... \times A_s^D , \mathbb{Z}/\ell^s\mathbb{Z})$$
which maps $a_1 \otimes ... \otimes a_N$ to the multilinear map $f \in \text{Mult}_c(A_s^D \times ... \times A_s^D , \mathbb{Z}/\ell^s\mathbb{Z})$ defined as follows: if $g_1,...,g_N \in \text{Hom}_c(A_s,\mathbb{Z}/\ell^s\mathbb{Z})$, then $f(g_1,...,g_N)= g_1(a_1)...g_N(a_N) \in \mathbb{Z}/\ell^s\mathbb{Z}$. \\

Let's check that $\varphi_s'$ is an isomorphism. By the structure theorem of abelian groups with finite exponent, $A_s$ is a direct sum of cyclic groups. Write $A_s \cong \bigoplus_{i\in I} \mathbb{Z}/\ell^{\alpha_i}\mathbb{Z}$ with $\alpha_i >0$ for each $i$ and observe that, since $A$ is divisible, all the $\alpha_i$ have to be equal to $s$. Now define the morphism:
$$\psi_s:  (\mathbb{Z}/\ell^s\mathbb{Z})^{\otimes N} \rightarrow \text{Mult}_c((\mathbb{Z}/\ell^s\mathbb{Z})^D \times ... \times (\mathbb{Z}/\ell^s\mathbb{Z})^D , \mathbb{Z}/\ell^s\mathbb{Z})$$
which maps $b_1 \otimes ... \otimes b_N$ to the multilinear map $f \in \text{Mult}_c((\mathbb{Z}/\ell^s\mathbb{Z})^D \times ... \times (\mathbb{Z}/\ell^s\mathbb{Z})^D , \mathbb{Z}/\ell^s\mathbb{Z})$ defined as follows: if $g_1,...,g_N \in \text{Hom}_c(\mathbb{Z}/\ell^s\mathbb{Z},\mathbb{Z}/\ell^s\mathbb{Z})$, then $f(g_1,...,g_N)= g_1(b_1)...g_N(b_N) \in \mathbb{Z}/\ell^s\mathbb{Z}$. By the compatibility of the operation $\boxtimes$ with direct sums, it is enough to prove that $\psi_s$ is an isomorphism to deduce that $\varphi_s'$ is also an isomorphism. But the groups $(\mathbb{Z}/\ell^s\mathbb{Z})^{\otimes N}$ and $\text{Mult}_c((\mathbb{Z}/\ell^s\mathbb{Z})^D \times ... \times (\mathbb{Z}/\ell^s\mathbb{Z})^D , \mathbb{Z}/\ell^s\mathbb{Z})$ are both cyclic of order $\ell^s$ and $\psi_s$ maps the generator $1\otimes ... \otimes 1$ of $(\mathbb{Z}/\ell^s\mathbb{Z})^{\otimes N}$ to the generator $f_0$ of $(\mathbb{Z}/\ell^s\mathbb{Z})^{\boxtimes N}$ defined by $f_0(\text{id}_{\mathbb{Z}/\ell^s\mathbb{Z}},...,\text{id}_{\mathbb{Z}/\ell^s\mathbb{Z}})= 1 \in \mathbb{Z}/\ell^s\mathbb{Z}$. This proves that both $\psi_s$ and $\varphi_s'$ are isomorphisms. By composing $\varphi_s'$ with the isomorphism (\ref{isom}), we get a natural isomorphism:
$$\varphi_s: (A_s)^{\otimes N} \rightarrow (A_s)^{\boxtimes N}.$$

All that remains to prove is the commutativity of diagram (\ref{commut}). For this purpose, it is enough to prove that the diagram:
\begin{equation}\label{diag2}
\xymatrix{
\text{Mult}_c(A_s^D \times ... \times A_s^D , \mathbb{Z}/\ell^s\mathbb{Z})\ar[d]^{j_{s,t}} & A_s^{\otimes N}  \ar[l]_-{\varphi_s'}^-{\cong} \ar[d]^{f_{s,t}}  \\
\text{Mult}_c(A_t^D \times ... \times A_t^D , \mathbb{Z}/\ell^t\mathbb{Z})  &  A_t^{\otimes N}  \ar[l]_-{\varphi_t'}^-{\cong} 
}
\end{equation}
commutes. Here, the left vertical morphism $j_{s,t}$ is induced by the injection $A_s \hookrightarrow A_t$ and by the injection $\psi_{s,t}:\mathbb{Z}/\ell^s\mathbb{Z} \hookrightarrow \mathbb{Z}/\ell^r\mathbb{Z}$ sending 1 to $\ell^{r-s}$. \\

Take any $a_1,...,a_N \in A_s$ and choose $b_1,...,b_N \in A_t$ such that $\ell^{t-s}b_i=a_i$. We then compute:
\begin{align*}
j_{s,t}(\varphi_s'(a_1\otimes ... \otimes a_N)):&(g_1,...,g_N)\in A_t^D \times ... \times A_t^D \mapsto  \psi_{s,t}(\psi_{s,t}^{-1}(g_1(a_1))...\psi_{s,t}^{-1}(g_N(a_N))),\\
\varphi_t'(f_{s,t}(a_1\otimes ... \otimes a_N)):&(g_1,...,g_N)\in A_t^D \times ... \times A_t^D \mapsto \ell^{t-s}g_1(b_1)...g_N(b_N).
\end{align*}
One easily checks that, if $c_1,...,c_N$ are elements in $\mathbb{Z}/\ell^t\mathbb{Z}$, then:
$$\psi_{s,t}(\psi_{s,t}^{-1}(\ell^{t-s}c_1)...\psi_{s,t}^{-1}(\ell^{t-s}c_N)) = \ell^{t-s}c_1...c_N.$$
Hence $j_{s,t}\circ \varphi_s' = \varphi_t' \circ f_{s,t}$, and the commutativity of (\ref{diag2}) is proved.\\
\end{proof}

\begin{theorem}\label{surj}
Let $\ell$ be a prime number and let $K$ be a field with characteristic different from $\ell$. Let $N$ be a positive integer. Then there exists a surjection of abelian groups:
$$H^1(K,\mathbb{Q}_{\ell}/\mathbb{Z}_{\ell}(1))^{\boxtimes N} \rightarrow H^N(K,\mathbb{Q}_{\ell}/\mathbb{Z}_{\ell}(N)).$$
\end{theorem}

\begin{proof}[Proof]
Let $A $ be the abelian group $H^1(K,\mathbb{Q}_{\ell}/\mathbb{Z}_{\ell}(1))$ and set $A_s = {_{\ell^s}}A$ for each $s\geq 1$. One has an exact sequence:
$$H^0(K,\mathbb{Q}_{\ell}/\mathbb{Z}_{\ell}(1))\xrightarrow[]{\delta^0_s}  H^1(K,\mathbb{Z}/\ell^s\mathbb{Z}(1)) \rightarrow  A \xrightarrow[]{\cdot \ell^s}  A \rightarrow H^2(K,\mathbb{Z}/\ell^s\mathbb{Z}(1))\rightarrow  H^2(K,\mathbb{Q}_{\ell}/\mathbb{Z}_{\ell}(1)).$$
In this sequence, the group $ H^2(K,\mathbb{Q}_{\ell}/\mathbb{Z}_{\ell}(1))$ can be identified with the Brauer group of $K$. We deduce that the morphism $H^2(K,\mathbb{Z}/\ell^s\mathbb{Z}(1)) \rightarrow H^2(K,\mathbb{Q}_{\ell}/\mathbb{Z}_{\ell}(1))$ is injective, and we get an exact sequence:
\begin{equation}\label{kum}
H^0(K,\mathbb{Q}_{\ell}/\mathbb{Z}_{\ell}(1))\xrightarrow[]{\delta^0_s}  H^1(K,\mathbb{Z}/\ell^s\mathbb{Z}(1))\rightarrow  A \xrightarrow[]{\cdot \ell^s}  A\rightarrow 0.
\end{equation}
This implies that the group $ A$ is divisible and that there are natural isomorphisms: $$A_s\cong H^1(K,\mathbb{Z}/\ell^s\mathbb{Z}(1))/\mathrm{Im}(\delta_s^0 ).$$ Hence, lemma \ref{torsbis} provides natural isomorphisms:
$$\varphi_s: (A_s)^{\otimes N} \rightarrow (A_s)^{\boxtimes N}$$
for $s\geq 1$ as well as morphisms $i_{s,t}: (A_s)^{\boxtimes N} \rightarrow (A_t)^{\boxtimes N}$ and $f_{s,t}: (A_s)^{\otimes N} \rightarrow (A_t)^{\otimes N}$ for $t\geq s$.\\

Now observe that the Bloch-Kato conjecture shows that the morphism induced by the cup-product:
$$\cup:\left[ H^1(K,\mathbb{Z}/\ell^s\mathbb{Z}(1)) \right] ^{\otimes N} \twoheadrightarrow H^N(K,\mathbb{Z}/\ell^s\mathbb{Z}(N))$$
 is a surjective morphism. Since the cup-product is compatible with the boundary maps
\begin{gather*}
\delta^0_s: H^0(K,\mathbb{Q}_{\ell}/\mathbb{Z}_{\ell}(1)) \rightarrow H^1(K,\mathbb{Z}/\ell^s\mathbb{Z}(1))\\
d^{N-1}_s: H^{N-1}(K,\mathbb{Q}_{\ell}/\mathbb{Z}_{\ell}(N)) \rightarrow H^N(K,\mathbb{Z}/\ell^s\mathbb{Z}(N)),
\end{gather*} 
this induces a surjective morphism:
 $$(A_s) ^{\otimes N} \twoheadrightarrow H^N(K,\mathbb{Z}/\ell^s\mathbb{Z}(N))/\mathrm{Im}(d^{N-1}_s)$$
 which will still be denoted by $\cup$. Hence, for $t\geq s$ we have a diagram:
 \begin{equation}\label{diag3}
\xymatrix{
{_{\ell^s}}(A^{\boxtimes N}) \ar[r]^{\cong} \ar@{^{(}->}[d]& (A_s)^{\boxtimes N}\ar[d]^{i_{s,t}} & (A_s)^{\otimes N}  \ar[l]_-{\varphi_s}^-{\cong} \ar@{->>}[r]^-{\cup}\ar[d]^{f_{s,t}}\commutatif  & H^N(K,\mathbb{Z}/\ell^s\mathbb{Z}(N))/\mathrm{Im}(d^{N-1}_s)\ar[d]^{j_{s,t}}\\
{_{\ell^t}}(A^{\boxtimes N}) \ar[r]^{\cong} &(A_t)^{\boxtimes N}  &  (A_t)^{\otimes N}  \ar[l]_-{\varphi_t}^-{\cong} \ar@{->>}[r]^-{\cup} & H^N(K,\mathbb{Z}/\ell^t\mathbb{Z}(N))/\mathrm{Im}(d^{N-1}_t)
}
\end{equation}
in which: 
\begin{itemize}
\item[$\bullet$] the right vertical morphism $j_{s,t}$ is induced by the morphism of Galois modules $h_{s,t}: \mathbb{Z}/\ell^s\mathbb{Z}(N) \rightarrow \mathbb{Z}/\ell^t\mathbb{Z}(N)$ defined by the formula $h_{s,t}(a_1\otimes ... \otimes a_N)=\ell^{t-s}(b_1\otimes ... \otimes b_N)$ where each $b_i \in \mathbb{Z}/\ell^t\mathbb{Z}(N)$ satisfies $\ell^{t-s}b_i=a_i$;
\item[$\bullet$] the left square is commutative by lemma \ref{tors} and the middle square is commutative by lemma \ref{torsbis}.
\end{itemize}
~\\
Let's now prove the commutativity of the square $(\star)$ of diagram (\ref{diag3}). For this purpose, observe that $(\star)$ can be lifted to a diagram:
 \begin{equation}\label{diag4}
\xymatrix{
 H^1(K,\mathbb{Z}/\ell^s\mathbb{Z}(1))^{\otimes N} \ar@{->>}[r]^-{\cup}\ar[d]^{\tilde{f}_{s,t}}  & H^N(K,\mathbb{Z}/\ell^s\mathbb{Z}(N))\ar[d]^{\tilde{j}_{s,t}}\\
  H^1(K,\mathbb{Z}/\ell^t\mathbb{Z}(1))^{\otimes N}   \ar@{->>}[r]^-{\cup} & H^N(K,\mathbb{Z}/\ell^t\mathbb{Z}(N))
}
\end{equation}
in which:
\begin{itemize}
\item[$\bullet$] $\tilde{j}_{s,t}$ is induced by  the morphism of Galois modules $h_{s,t}: \mathbb{Z}/\ell^s\mathbb{Z}(N) \rightarrow \mathbb{Z}/\ell^t\mathbb{Z}(N)$;
\item[$\bullet$] $\tilde{f}_{s,t}$ is defined as follows: for $a_1, ... , a_N \in H^1(K,\mathbb{Z}/\ell^s\mathbb{Z}(1))=K^{\times}/(K^{\times})^{\ell^s}$, one sets $$\tilde{f}_{s,t}(a_1\otimes ... \otimes a_N) =\ell^{t-s}([\tilde{a}_1]\otimes ... \otimes [\tilde{a}_N])\in H^1(K,\mathbb{Z}/\ell^t\mathbb{Z}(1))^{\otimes N}$$ where each $\tilde{a}_i$ is an element of $K^{\times}$ that lifts $a_i$ and $[\tilde{a}_i]$ denotes the classe of $\tilde{a}_i$ in $H^1(K,\mathbb{Z}/\ell^t\mathbb{Z}(1))=K^{\times}/(K^{\times})^{\ell^t}$.
\end{itemize}
 Hence we only need to prove the commutativity of diagram (\ref{diag4}). \\

To do so, take any $a_1,...,a_N \in H^1(K,\mathbb{Z}/\ell^s\mathbb{Z}(1))$ and let $\tilde{a}_1,...,\tilde{a}_N \in K^{\times}$ be liftings of the $a_i$'s. For each $i\in \{1,...,N\}$, choose $\rho_i \in K^s$ an $\ell^t$-th root of $\tilde{a}_i$ and define the \textit{homogeneous} cocycle:
\begin{align*}
\beta_i: \gal (K^s/K)^2 &\rightarrow \mathbb{Z}/\ell^t\mathbb{Z}(1),\\
(\sigma_0,\sigma_1) & \mapsto \frac{\sigma_0(\rho_i)}{\sigma_1(\rho_i)}.
\end{align*}
 Then the cohomology class $[\tilde{a}_i] \in H^1(K,\mathbb{Z}/\ell^t\mathbb{Z}(1))$ is represented by $\beta_i$ and the cohomology class $a_i$ is represented by the homogeneous cocycle:
\begin{align*}
\alpha_i: \gal (K^s/K)^2 &\rightarrow \mathbb{Z}/\ell^s\mathbb{Z}(1),\\ 
(\sigma_0,\sigma_1) & \mapsto \beta_i(\sigma_0,\sigma_1)^{\ell^{t-s}}.
\end{align*}
We deduce that $\tilde{j}_{s,t}(a_1\cup ... \cup a_N)$ and $\cup (\tilde{f}_{s,t}(a_1\otimes ... \otimes a_N))$ are both represented by the homogeneous cocycle:
$$(\sigma_0,...,\sigma_N) \in  \gal (K^s/K)^{N+1} \rightarrow \ell^{t-s}\beta_1(\sigma_0,\sigma_1)\otimes\beta_2(\sigma_1,\sigma_2)\otimes... \otimes \beta_N(\sigma_{N-1},\sigma_N) \in \mathbb{Z}/\ell^t\mathbb{Z}(1).$$
Hence, $\tilde{j}_{s,t}(a_1\cup ... \cup a_N) = \cup (\tilde{f}_{s,t}(a_1\otimes ... \otimes a_N))$, and diagram (\ref{diag4}) commutes. By passing the surjective morphism:
$${_{\ell^s}}(A^{\boxtimes N}) \xrightarrow[]{\cong} (A_s)^{\boxtimes N} \xrightarrow[]{\varphi_s^{-1}} (A_s)^{\otimes N}  \twoheadrightarrow H^N(K,\mathbb{Z}/\ell^s\mathbb{Z}(N))/\mathrm{Im}(d^{N-1}_s)$$
 to the direct limit on $s$, we get the desired surjection of Galois modules:
$$H^1(K,\mathbb{Q}_{\ell}/\mathbb{Z}_{\ell}(1))^{\boxtimes N} \rightarrow H^N(K,\mathbb{Q}_{\ell}/\mathbb{Z}_{\ell}(N)),$$
because $\varinjlim_s H^{N-1}(K,\mathbb{Q}_{\ell}/\mathbb{Z}_{\ell}(N))/\ell^s = 0$ and hence $\varinjlim_s \mathrm{Im}(d^{N-1}_s) = 0$.
\end{proof}

\section{Vanishing theorems in Galois cohomology}\label{sec1}

\subsection{An abstract vanishing theorem}\label{pf}

\hspace{3ex} In this section we fix a perfect field $k$ and a non-negative integer $M$. We also consider a commutative, local, normal, henselian, excellent, $M$-dimensional ring $R$ with residue field $k$. For instance, $R$ could be the henselization or the completion of the local ring at a closed point of a normal $k$-variety or, more generally, of a normal finite type scheme defined over an excellent ring. Note that the ring $R$ may have mixed characteristic. \\

\hspace{3ex} In the sequel, we let $R^{sh}$ be the strict henselization of $R$ and we set $\mathcal{X} = \text{Spec} \; R$ and $\mathcal{X}^{sh} = \text{Spec} \; R^{sh}$. Note that the schemes $\mathcal{X}$ and $\mathcal{X}^{sh}$ may be singular.

\begin{theorem}\label{abs1}
Let $d\geq 1$ and $r$ be two integers and $\ell$ a prime number different from the characteristic of $k$. Let $f:\mathcal{Y} \rightarrow \mathcal{X}$ be a proper dominant morphism such that the scheme $\mathcal{Y}^{sh}:= \mathcal{Y} \times_{\mathcal{X}} \mathcal{X}^{sh}$ is integral and regular. Let $N$ be the dimension of the generic fiber of $f$ and make the following two assumptions:
\begin{itemize}
\item[(H1)] The field $k$ has $\ell$-cohomological dimension $d$.
\item[(H2)] For each $j\in \{0,...,M+N\}$, each finite extension $k'$ of $k$ and each semi-abelian variety $G$ over $k'$, one has: $$H^d(k',G(k^s)\{\ell\}^{\boxtimes j}(r-M-N))=0.$$
\end{itemize}
If $K$ is the function field of $\mathcal{Y}$, then the group $H^{d+M+N}(K,\mathbb{Q}_{\ell}/\mathbb{Z}_{\ell}(r))$ vanishes. Moreover, if the special fiber $Y$ of $f$ is smooth over $k$, then assumption (H2) can be replaced by:
\begin{itemize}
\item[(H2')] For each $j\in \{0,...,M+N\}$, each finite extension $k'$ of $k$ and each abelian variety $A$ over $k'$, one has: $$H^d(k',A(k^s)\{\ell\}^{\boxtimes j}(r-M-N))=0.$$
\end{itemize}
\end{theorem}

\begin{remarque}~
\begin{itemize}
\item[(i)] In the case $\mathcal{Y}$ is not assumed to be regular but has a desingularization $\tilde{\mathcal{Y}}$, the theorem can be applied to $\tilde{\mathcal{Y}}$ and hence the conclusion still holds. In particular, if $k$ has characteristic 0, then the previous theorem holds when $K$ is the function field of any integral variety defined over the fraction field of $R$.
\item[(ii)] When $M=0$, we have $R=k$ and $Y$ is a smooth $k$-variety. Hence theorem \ref{abs1} covers the case when $K$ is the function field of a smooth projective $k$-variety. Also note that, in this situation, assumption (H2') is always enough.
\item[(iii)] When $N=0$, the schemes $\mathcal{Y}$  and $ \mathcal{X}$ are birational. Hence theorem \ref{abs1} covers the case when $K$ is the fraction field of $R$.
\end{itemize}
\end{remarque}

\hspace{3ex} The first step in the proof of the theorem consists in describing the structure of the Galois module $H^1(k^s(\mathcal{Y}^{sh}),\mathbb{Q}_{\ell}/\mathbb{Z}_{\ell}(1))$, where $k^s(\mathcal{Y}^{sh})$ stands for the function field of $\mathcal{Y}^{sh}$. For this purpose, we first study the Picard group of $ \mathcal{Y}^{sh}$.

\begin{lemma}\label{semiab} ~
\begin{itemize}
\item[(i)] There exist a semi-abelian variety $G$, a finite Galois module $\Phi$ and an exact sequence of Galois modules:
\begin{gather*}
0\rightarrow G(k^s)\{\ell\} \rightarrow (\text{Pic}\; \mathcal{Y}^{sh})\{\ell\} \rightarrow \Phi \rightarrow 0.
\end{gather*}
Moreover, if the special fiber $Y$ of $f:\mathcal{Y}\rightarrow \mathcal{X}$ is smooth, then $G$ is an abelian variety.
\item[(ii)] The morphism $(\text{Pic}\; \mathcal{Y}^{sh})\otimes \mathbb{Q}_{\ell}/\mathbb{Z}_{\ell} \rightarrow (\text{Pic}\; Y_{k^s})\otimes \mathbb{Q}_{\ell}/\mathbb{Z}_{\ell}$ is injective.
\end{itemize}

\end{lemma}

\begin{proof}[Proof]
For each $m\geq 1$, we have commutative diagrams with exact lines:
\begin{equation}
\xymatrix{
0\ar[r] &  \mathcal{O}(\mathcal{Y}^{sh})^{\times}/{\mathcal{O}(\mathcal{Y}^{sh})^{\times}}^{\ell^m}\ar[d] \ar[r] & H^1(\mathcal{Y}^{sh},\mu_{\ell^m}) \ar[r]\ar[d] & {_{\ell^m}}\text{Pic}\; \mathcal{Y}^{sh} \ar[r]\ar[d] & 0\\
0\ar[r] &  \mathcal{O}(Y_{k^s})^{\times}/{\mathcal{O}(Y_{k^s})^{\times}}^{\ell^m} \ar[r] & H^1(Y_{k^s},\mu_{\ell^m}) \ar[r] & {_{\ell^m}}\text{Pic}\; Y_{k^s} \ar[r] & 0,
}
\end{equation}
\begin{equation}
\xymatrix{
0\ar[r] & (\text{Pic}\; \mathcal{Y}^{sh})/\ell^m \ar[d] \ar[r] & H^2(\mathcal{Y}^{sh},\mu_{\ell^m}) \ar[r]\ar[d] & {_{\ell^m}}\text{Br}\; \mathcal{Y}^{sh} \ar[r]\ar[d] & 0\\
0\ar[r] &  (\text{Pic}\; Y_{k^s})/\ell^m \ar[r] & H^2(Y_{k^s},\mu_{\ell^m}) \ar[r] & {_{\ell^m}}\text{Br}\; Y_{k^s} \ar[r] & 0.
}
\end{equation}
In both diagrams, the middle vertical morphism is an isomorphism by the proper base change theorem. Moreover, the group $\mathcal{O}(Y_{k^s})^{\times}$ is divisible, so that $\mathcal{O}(Y_{k^s})^{\times}/{\mathcal{O}(Y_{k^s})^{\times}}^{\ell^m}=0$. A simple diagram chase then shows that ${_{\ell^m}}\text{Pic}\; \mathcal{Y}^{sh}\cong {_{\ell^m}}\text{Pic}\; Y_{k^s} $ and that the morphism $(\text{Pic}\; \mathcal{Y}^{sh})/\ell^m \rightarrow (\text{Pic}\; Y_{k^s})/\ell^m$ is injective. Hence:
$$(\text{Pic}\; \mathcal{Y}^{sh})\{\ell\}\cong (\text{Pic}\; Y_{k^s})\{\ell\} $$
and the morphism $(\text{Pic}\; \mathcal{Y}^{sh})\otimes \mathbb{Q}_{\ell}/\mathbb{Z}_{\ell} \rightarrow (\text{Pic}\; Y_{k^s})\otimes \mathbb{Q}_{\ell}/\mathbb{Z}_{\ell}$ is injective.\\

Now note that $\text{Pic}^0\; Y_{k^s} = G_0(k^s)$ for some connected commutative algebraic group $G_0$. If $Y$ is smooth, then $G_0$ is an abelian variety, while in general, by the structure theorem of connected commutative algebraic groups, we have an exact sequence:
$$0\rightarrow U \rightarrow G_0 \rightarrow G \rightarrow 0$$
for some semi-abelian variety $G$ and unipotent group $U$. Since the group $U(k^s)$ is uniquely $\ell$-divisible, we get an isomorphism:
$$(\text{Pic}^0\; Y_{k^s})\{\ell\} \cong G(k^s)\{\ell\}$$
and the group $\text{Pic}^0\; Y_{k^s}$ is $\ell$-divisible. The lemma follows by noting that the Néron-Severi group $NS(Y_{k^s}) = \text{Pic}\; Y_{k^s}/\text{Pic}^0\; Y_{k^s}$ is of finite type.
\end{proof}

\begin{lemma}\label{suites0}
There exist a finite subset $S$ of $ (\mathcal{Y}^{sh})^{(1)}$ and Galois modules $Q$, $Q'$, $F$ and $N$ over $k$ such that there are exact sequences:
\begin{gather}\label{picdiv2}
0 \rightarrow \mathrm{Pic}(\mathcal{Y}^{sh})\{\ell\} \rightarrow H^1(k^s(\mathcal{Y}^{sh}),\mathbb{Q}_{\ell}/\mathbb{Z}_{\ell}(1)) \rightarrow Q' \rightarrow 0,\\
\label{exq}
0 \rightarrow Q \otimes \mathbb{Q}_{\ell}/\mathbb{Z}_{\ell} \rightarrow Q' \rightarrow F \rightarrow 0,\\
\label{propq2}
0 \rightarrow N\otimes \mathbb{Q}_{\ell}/\mathbb{Z}_{\ell} \rightarrow Q\otimes \mathbb{Q}_{\ell}/\mathbb{Z}_{\ell} \twoheadrightarrow \bigoplus_{v\in (\mathcal{Y}^{sh})^{(1)} \setminus S} \mathbb{Q}_{\ell}/\mathbb{Z}_{\ell}\rightarrow 0,
\end{gather}
and such that $Q$ is a free abelian group, $F$ is finite, $N$ is a finite type abelian group and $S$ is stable under the action of the absolute Galois group of $k$.
\end{lemma}

\begin{proof}[Proof]
Since $\mathcal{Y}^{sh}$ is regular, we have an exact sequence of Galois modules:

\begin{equation*}
0 \rightarrow \mathcal{O}(\mathcal{Y}^{sh})^{\times} \rightarrow k^s(\mathcal{Y}^{sh})^{\times} \rightarrow \mathrm{Div}(\mathcal{Y}^{sh}) \rightarrow \mathrm{Pic}(\mathcal{Y}^{sh})\rightarrow 0.
\end{equation*}

Hence, by the snake lemma, for each positive integer $m$, we have an exact sequence:
\begin{equation}\label{000}
0 \rightarrow {_{\ell^m}}\text{Pic}(\mathcal{Y}^{sh}) \rightarrow H^1(k^s(\mathcal{Y}^{sh}),\mathbb{Z}/\ell^m\mathbb{Z}(1)) \rightarrow \text{Div}(\mathcal{Y}^{sh})/\ell^m \rightarrow \text{Pic}(\mathcal{Y}^{sh})/\ell^m \rightarrow 0.
\end{equation}

By passing to the direct limit on $m$, we obtain the following exact sequence of Galois modules over $k$:
\begin{equation}\label{picdiv}
0 \rightarrow \mathrm{Pic}(\mathcal{Y}^{sh})\{\ell\} \rightarrow H^1(k^s(\mathcal{Y}^{sh}),\mathbb{Q}_{\ell}/\mathbb{Z}_{\ell}(1)) \rightarrow \mathrm{Div}(\mathcal{Y}^{sh}) \otimes \mathbb{Q}_{\ell}/\mathbb{Z}_{\ell} \rightarrow \mathrm{Pic}(\mathcal{Y}^{sh}) \otimes \mathbb{Q}_{\ell}/\mathbb{Z}_{\ell} \rightarrow 0.
\end{equation}
Consider the following Galois modules:
\begin{align*}
Q&:= \text{Ker}\left( \mathrm{Div}(\mathcal{Y}^{sh}) \rightarrow \text{NS}(Y_{k^s})/ \text{NS}(Y_{k^s})_{\mathrm{tors}} \right);\\
I&:= \text{Im}\left( \mathrm{Div}(\mathcal{Y}^{sh}) \rightarrow \text{NS}(Y_{k^s})/ \text{NS}(Y_{k^s})_{\mathrm{tors}} \right);\\
Q'&:=\text{Ker}\left( \mathrm{Div}(\mathcal{Y}^{sh})\otimes \mathbb{Q}_{\ell}/\mathbb{Z}_{\ell} \rightarrow \mathrm{Pic}(\mathcal{Y}^{sh}) \otimes \mathbb{Q}_{\ell}/\mathbb{Z}_{\ell} \right).
\end{align*}
Exact sequence (\ref{picdiv2}) immediately follows from the definition of $Q'$ and from exact sequence (\ref{picdiv}).\\

Let's now construct exact sequence (\ref{exq}). Note that, by lemma \ref{semiab}(ii), the morphism $$(\text{Pic}\; \mathcal{Y}^{sh})\otimes \mathbb{Q}_{\ell}/\mathbb{Z}_{\ell} \rightarrow (\text{Pic}\; Y_{k^s})\otimes \mathbb{Q}_{\ell}/\mathbb{Z}_{\ell} \cong (\text{NS}\; Y_{k^s})\otimes \mathbb{Q}_{\ell}/\mathbb{Z}_{\ell}$$ is injective, so that we have an exact sequence:
\begin{equation}\label{q'}
0\rightarrow Q' \rightarrow \mathrm{Div}(\mathcal{Y}^{sh})\otimes \mathbb{Q}_{\ell}/\mathbb{Z}_{\ell} \rightarrow (\text{NS}\; Y_{k^s})\otimes \mathbb{Q}_{\ell}/\mathbb{Z}_{\ell}.
\end{equation}
Moreover, by definition of $Q$ and $I$, we have the following exact sequences:
\begin{align}\label{defq}
0 \rightarrow Q \rightarrow \mathrm{Div}(\mathcal{Y}^{sh}) \rightarrow I \rightarrow 0,\\
0 \rightarrow I \rightarrow \text{NS}(Y_{k^s})/ \text{NS}(Y_{k^s})_{\mathrm{tors}}\rightarrow I' \rightarrow 0,
\end{align}
for some Galois module $I'$. Since the abelian group $\text{NS}(Y_{k^s})/ \text{NS}(Y_{k^s})_{\mathrm{tors}}$ is of finite type and has no torsion, we get new exact sequences by tensoring with $\mathbb{Q}_{\ell}/\mathbb{Z}_{\ell}$:
\begin{align}\label{defq2}
0 = \text{Tor}^1(I,\mathbb{Q}_{\ell}/\mathbb{Z}_{\ell}) \rightarrow Q\otimes \mathbb{Q}_{\ell}/\mathbb{Z}_{\ell} \rightarrow \mathrm{Div}(\mathcal{Y}^{sh})\otimes \mathbb{Q}_{\ell}/\mathbb{Z}_{\ell} \rightarrow I\otimes \mathbb{Q}_{\ell}/\mathbb{Z}_{\ell} \rightarrow 0,\\
\label{defq23}
\text{Tor}^1(I',\mathbb{Q}_{\ell}/\mathbb{Z}_{\ell}) \rightarrow I\otimes \mathbb{Q}_{\ell}/\mathbb{Z}_{\ell} \rightarrow \text{NS}(Y_{k^s})\otimes \mathbb{Q}_{\ell}/\mathbb{Z}_{\ell} \rightarrow I'\otimes \mathbb{Q}_{\ell}/\mathbb{Z}_{\ell} \rightarrow 0,
\end{align}
in which the group $\text{Tor}^1(I',\mathbb{Q}_{\ell}/\mathbb{Z}_{\ell})$ is finite. 
One then only has to combine exact sequences (\ref{q'}), (\ref{defq2}) and (\ref{defq23}) to obtain exact sequence (\ref{exq}). \\

Let's finish the proof by constructing exact sequence (\ref{propq2}). By using once again the fact that $\text{NS}(Y_{k^s})$ is a finite type abelian group, one can find a finite subset $S$ of $(\mathcal{Y}^{sh})^{(1)}$ such that $Q$ maps onto $\bigoplus_{v\in (\mathcal{Y}^{sh})^{(1)} \setminus S} \mathbb{Z}$. Up to replacing $S$ by its orbit under the action of the Galois group of $k$, one can assume that $S$ is Galois-equivariant, and one therefore gets a surjection of Galois modules:
$$Q \twoheadrightarrow \bigoplus_{v\in (\mathcal{Y}^{sh})^{(1)} \setminus S} \mathbb{Z}.$$
Its kernel is a Galois module $N$ which, as an abelian group, is finitely generated. Hence one gets an exact sequence of Galois modules:
\begin{equation}\label{propq}
0 \rightarrow N \rightarrow Q \rightarrow \bigoplus_{v\in (\mathcal{Y}^{sh})^{(1)} \setminus S} \mathbb{Z}\rightarrow 0.
\end{equation}
Exact sequence (\ref{propq2}) is then obtained by tensoring with $\mathbb{Q}_{\ell}/\mathbb{Z}_{\ell}$.
\end{proof}

\begin{proposition}\label{expfini}
The group $H^d(k,H^1(k^s(\mathcal{Y}^{sh}),\mathbb{Q}_{\ell}/\mathbb{Z}_{\ell}(1))^{\boxtimes M+N}(r-M-N))$ has finite exponent.
\end{proposition}

\begin{proof}[Proof]
By lemmas \ref{semiab} and \ref{suites0}, we have the following exact sequences of Galois modules:
\begin{gather}\label{picdiv3}
0 \rightarrow G(k^s)\{\ell\} \rightarrow H^1(k^s(\mathcal{Y}^{sh}),\mathbb{Q}_{\ell}/\mathbb{Z}_{\ell}(1)) \rightarrow Q'' \rightarrow 0,\\
\label{autre}
0 \rightarrow \Phi \rightarrow Q'' \rightarrow Q'  \rightarrow 0\\
\label{autre2}
0 \rightarrow Q\otimes \mathbb{Q}_{\ell}/\mathbb{Z}_{\ell} \rightarrow Q' \rightarrow F  \rightarrow 0\\
\label{propq3}
0 \rightarrow N\otimes \mathbb{Q}_{\ell}/\mathbb{Z}_{\ell} \rightarrow Q\otimes \mathbb{Q}_{\ell}/\mathbb{Z}_{\ell} \rightarrow \bigoplus_{v\in (\mathcal{Y}^{sh})^{(1)} \setminus S} \mathbb{Q}_{\ell}/\mathbb{Z}_{\ell}\rightarrow 0,
\end{gather}
for some Galois module $Q''$. By assumption (H2), one can apply lemma \ref{abstractdiv} to the sequences (\ref{propq3}) and (\ref{autre2}) and one gets that the group $$H^d(k,(G(k^s)\{\ell\})^{\boxtimes j} \boxtimes Q'^{\boxtimes M+N-j}(r-M-N))$$ has finite exponent for $j=0,1,...,M+N$. One can then apply lemma \ref{abstractfini} to the sequence (\ref{autre}) and one gets that the group
$$H^d(k,(G(k^s)\{\ell\})^{\boxtimes j} \boxtimes (Q'')^{\boxtimes M+N-j}(r-M-N))$$
has finite exponent. One can finally apply a second time lemma \ref{abstractdiv} to the sequence (\ref{picdiv3}) and one gets that the group
$$H^d(k,H^1(k^s(\mathcal{Y}),\mathbb{Q}_{\ell}/\mathbb{Z}_{\ell}(1))^{\boxtimes M+N}(r-M-N))$$
has finite exponent.
\end{proof}

\begin{proof}[Proof of theorem \ref{abs1}]
Write the Hochschild-Serre spectral sequence: $$H^s(k,H^t(k^s(\mathcal{Y}^{sh}),\mathbb{Q}_{\ell}/\mathbb{Z}_{\ell}(r))) \Rightarrow H^{s+t}(K,\mathbb{Q}_{\ell}/\mathbb{Z}_{\ell}(r)).$$
Since $k$ has $\ell$-cohomological dimension $d$ (assumption (H1)), the previous spectral sequence induces an isomorphism: 
\begin{align*}
H^{d+M+N}(K,\mathbb{Q}_{\ell}/\mathbb{Z}_{\ell}(r))&\cong H^d(k,H^{M+N}(k^s(\mathcal{Y}^{sh}),\mathbb{Q}_{\ell}/\mathbb{Z}_{\ell}(r))) \\&\cong H^d(k,H^{M+N}(k^s(\mathcal{Y}^{sh}),\mathbb{Q}_{\ell}/\mathbb{Z}_{\ell}(M+N))(r-M-N)).
\end{align*}
By theorem \ref{surj}, one has a surjective morphism of Galois modules:
$$H^1(k^s(\mathcal{Y}^{sh}),\mathbb{Q}_{\ell}/\mathbb{Z}_{\ell}(1))^{\boxtimes M+N} \twoheadrightarrow H^{M+N}(k^s(\mathcal{Y}^{sh}),\mathbb{Q}_{\ell}/\mathbb{Z}_{\ell}(M+N)).$$
Since the field $k$ has $\ell$-cohomological dimension $d$ (assumption (H1)), one gets a surjective morphism:
$$H^d(k,H^1(k^s(\mathcal{Y}),\mathbb{Q}_{\ell}/\mathbb{Z}_{\ell}(1))^{\boxtimes M+N}(r-M-N)) \twoheadrightarrow H^{d+M+N}(K,\mathbb{Q}_{\ell}/\mathbb{Z}_{\ell}(r)).$$
But the group $H^d(k,H^1(k^s(\mathcal{Y}),\mathbb{Q}_{\ell}/\mathbb{Z}_{\ell}(1))^{\boxtimes M+N}(r-M-N))$ has finite exponent (proposition \ref{expfini}) and the group $H^{d+M+N}(K,\mathbb{Q}_{\ell}/\mathbb{Z}_{\ell}(r))$ is divisible because $K$ has $\ell$-cohomological dimension $d+M+N$ (assumption (H1)). One deduces that $H^{d+M+N}(K,\mathbb{Q}_{\ell}/\mathbb{Z}_{\ell}(r))=0$.

\end{proof}

\begin{corollary}\label{abs1bis}
Let $d\geq 1$ and $r$ be two integers and $\ell$ a prime number different from the characteristic of $k$. Let $f:\mathcal{Y} \rightarrow \mathcal{X}$ be a proper dominant morphism such that $\mathcal{Y}^{sh}:= \mathcal{Y} \times_{\mathcal{X}} \mathcal{X}^{sh}$ is integral and regular. Let $N$ be the dimension of the generic fiber of $f$ and make the following two assumptions:
\begin{itemize}
\item[(H1)] The field $k$ has $\ell$-cohomological dimension $d$.
\item[(H2)] For each $j_1,j_2\in \{0,...,M+N\}$ such that $j_1+j_2 \leq M+N$, each finite extension $k'$ of $k$ and each abelian variety $A$ over $k'$, one has: $$H^d(k',A(k^s)\{\ell\}^{\boxtimes j_1}(r-M-N+j_2))=0.$$
\end{itemize}
If $K$ is the function field of $\mathcal{Y}$, then the group $H^{d+M+N}(K,\mathbb{Q}_{\ell}/\mathbb{Z}_{\ell}(r))$ vanishes. 
\end{corollary}

\begin{proof}[Proof]
Take $G$ a semi-abelian variety over a finite extension $k'$ of $k$. By applying lemma \ref{abstractdiv} to an exact sequence of the form:
$$0 \rightarrow T \rightarrow G \rightarrow A \rightarrow 0$$
where $T$ is a torus and $A$ an abelian variety, one easily sees that assumption (H2) of corollary \ref{abs1bis} implies assumption (H2) of theorem \ref{abs1}.
\end{proof}

\subsection{Applications}\label{app1}

\hspace{3ex} In this section, we apply theorem \ref{abs1} to several concrete situations. 

\subsubsection{Finite base field}

\begin{lemma}\label{abfin}
Let $k$ be a finite field with $q$ elements. Let $A$ be an abelian variety over $k$ and let $\ell$ be a prime number different from the characteristic of $k$. Let $j\geq 0$ and $r$ be integers such that $j\neq -2r$. Then $H^1(k,A\{\ell\}^{\boxtimes j} (r))=0$.
\end{lemma}

\begin{remarque}
The vanishing of the lemma does not hold when $j=2r$.
\end{remarque}

\begin{proof}[Proof]
By duality over $k$ (example I.1.10 of \cite{milne}), one has:
\begin{align*}
H^1(k,A\{\ell\}^{\boxtimes j} (r)) & \cong \varinjlim_n  H^1(k,({_{\ell^n}}A) ^{\boxtimes j}(r)) \\
& \cong \left[ \varprojlim_n H^0(k,(({_{\ell^n}}A) ^{\boxtimes j})^D(-r))\right] ^D \\
& \cong \left[ \varprojlim_n H^0(k,(({_{\ell^n}}A)^D) ^{\otimes j}(-r))\right] ^D \\& \cong \left[\varprojlim_n H^0(k,({_{\ell^n}}A^t)^{\otimes j}(-j-r))\right] ^D \\& \cong H^0(k, (T_{\ell}A^t)^{\otimes j} \otimes \mathbb{Z}_{\ell}(-j-r))^D.
\end{align*}
According to the Weil conjectures, the eigenvalues of the geometric Frobenius on $(T_{\ell}A^t)^{\otimes j} \otimes \mathbb{Z}_{\ell}(-j-r)$ have modulus $q^{r+\frac{j}{2}} \neq 1$. Hence:
$$H^1(k,A\{\ell\}^{\boxtimes j} (r))=H^0(k, (T_{\ell}A^t)^{\otimes j} \otimes \mathbb{Z}_{\ell}(-j-r))^D=0.$$ 
\end{proof}

\begin{theorem}\label{thfini1}
Let $k$ be a finite field and let $\ell$ be a prime number different from the characteristic of $k$. Let $M$ be a non-negative integer and consider a proper morphism $f:\mathcal{Y} \rightarrow \mathcal{X}$ such that $\mathcal{X}$ is the spectrum of a commutative, local, normal, henselian, excellent, $M$-dimensional ring with residue field $k$. Let $\mathcal{X}^{sh}$ be the strict henselization of $\mathcal{X}$ and assume that $\mathcal{Y}^{sh}:= \mathcal{Y} \times_{\mathcal{X}} \mathcal{X}^{sh}$ is integral and regular. Let $K$ be the function field of $\mathcal{Y}$ and $N$ the dimension of the generic fiber of $f$.
\begin{itemize}
\item[(i)] The group $H^{M+N+1}(K,\mathbb{Q}_{\ell}/\mathbb{Z}_{\ell}(r))$ vanishes for all $r\not\in \left[ 0,M+N \right] $.
\item[(ii)] If the special fiber of $f$ is smooth, then the group $H^{M+N+1}(K,\mathbb{Q}_{\ell}/\mathbb{Z}_{\ell}(r))$ vanishes for all $r\not\in \left[ \frac{M+N}{2},M+N \right] $.
\end{itemize}
 
\end{theorem}

\begin{proof}[Proof]
One just has to combine theorem \ref{abs1} and corollary \ref{abs1bis} with lemma \ref{abfin}.
\end{proof}

\begin{remarque}
If $N=0$ and $K$ is therefore the function field of $\mathcal{X}$, then the group $H^{M+1}(K,\mathbb{Q}_{\ell}/\mathbb{Z}_{\ell}(r))$ vanishes for all $r\not\in \left[ 0,M \right] $.
\end{remarque}

\subsubsection{$p$-adic base field}

\hspace{3ex} Let $k$ be a field. Recall that a 1-motive over $k$ is a two-term complex of $k$-group schemes $M=[u: Y\rightarrow G]$ (placed in degrees -1 and 0), where $Y$ is the $k$-group associated to a free abelian group of finite type endowed with a continuous action of $\gal (k^s/k)$ and where $G$ is a semi-abelian variety, i.e. an extension of an abelian variety by a torus. When $M$ is a 1-motive over $k$, one can introduce the Galois module $T_{\mathbb{Z}/n\mathbb{Z}}(M):= H^{-1}(M\otimes^{\mathbf{L}} \mathbb{Z}/n\mathbb{Z})$ for each $n>0$. If $\ell$ is a prime number, one defines the $\ell$-adic Tate module of $M$ as follows:
$$T_{\ell}(M):= \varprojlim_n T_{\mathbb{Z}/\ell^n\mathbb{Z}}(M).$$

\begin{lemma}\label{motif}
Let $k$ be a $p$-adic field with residue field $\kappa$ and let $M=[u: Y\rightarrow G]$ be a 1-motive over $k$. Let $\ell$ be a prime number different from $p$. Let $T$ be a torus and $A$ an abelian variety such that $G$ is an extension of $A$ by $T$ and assume that $A$ has good reduction. Then $T_{\ell}(M)$ is a free $\mathbb{Z}_{\ell }$-module of finite type and the eigenvalues of a lifting $\phi \in \gal (k^s/k)$ of the geometric Frobenius $\text{Fr} \in \gal(\kappa^s/\kappa)$ acting on $T_{\ell}(M)$ have modulus 1, $q^{-1/2}$ or $q^{-1}$.
\end{lemma}

\begin{proof}[Proof]
For each $n\geq 1$, one has a distinguished triangle:
$$ Y/\ell^n \rightarrow {_{\ell^n}}G[1] \rightarrow M\otimes^{\mathbf{L}}\mathbb{Z}/\ell^n\mathbb{Z} \rightarrow Y/\ell^n[1].$$
One gets an exact sequence of finite Galois modules:
\begin{equation} \label{exmot}
0 \rightarrow {_{\ell^n}}G \rightarrow T_{\mathbb{Z}/\ell^n\mathbb{Z}}(M)\rightarrow Y/\ell^n \rightarrow 0.
\end{equation} 
Moreover, since $T(k^s)$ is divisible,, one also has the following exact sequence of finite Galois modules:
\begin{equation}\label{exsemi}
0 \rightarrow  {_{\ell^n}}T \rightarrow {_{\ell^n}}G \rightarrow {_{\ell^n}}A \rightarrow 0.
\end{equation}
Since the groups that are involved in the sequences (\ref{exmot}) and (\ref{exsemi}) are finite, by passing to the inverse limit on $n$, one gets the following exact sequences of Galois modules:
\begin{gather} \label{exmot2}
0 \rightarrow T_{\ell}G \rightarrow T_{\ell}M\rightarrow Y\otimes \mathbb{Z}_{\ell} \rightarrow 0,\\
\label{exsemi2}
0 \rightarrow  T_{\ell}T \rightarrow T_{\ell}G \rightarrow T_{\ell}A \rightarrow 0.
\end{gather}
This shows that $T_{\ell}(M)$ is a free $\mathbb{Z}_{\ell }$-module of finite type, so that it only remains to check the assertion about $\phi$. For this purpose, we study the action of $\phi$ on $Y\otimes \mathbb{Z}_{\ell}$, $T_{\ell}T$  and $T_{\ell}A$:
\begin{itemize}
\item[$\bullet$] since $Y$ becomes constant on a finite extension of $k$, all the eigenvalues of $\phi$ on $Y\otimes \mathbb{Z}_{\ell}$ have modulus 1.
\item[$\bullet$] since $T$ becomes isomorphic to $\mathbb{G}_m^r$ for some $r\geq 0$ on a finite extension of $k$, the eigenvalues of $\phi$ on $T_{\ell}T$ have modulus $q^{-1}$.
\item[$\bullet$] the abelian variety $A$ has good reduction. Let then $\mathcal{A}$ be an abelian scheme that extends $A$ on the ring of integers of $k$ and let $A_0$ be its special fiber. Since $\ell$ is different from $p$, we have $T_{\ell}A = T_{\ell}A_0$. We can therefore deduce from the Weil conjectures that the eigenvalues of $\phi$ on $T_{\ell}A$ all have modulus $q^{-1/2}$.
\end{itemize}
This finishes the proof.
\end{proof}

\begin{lemma}\label{abpadicell}
Let $k$ a $p$-adic field. Let $A$ be an abelian variety over $k$ and $\ell$ a prime number different from $p$. Let $j\geq 0$ and $r$ be integers such that $r\not\in \left[1-j,1 \right] $. Then $H^2(k,A\{\ell\}^{\boxtimes j} (r))=0$.
\end{lemma}

\begin{proof}[Proof]
By Tate's local duality theorem, one has:
\begin{align*}
H^2(k,A\{\ell\}^{\boxtimes j} (r)) & \cong \varinjlim_n  H^2(k,({_{\ell^n}}A) ^{\boxtimes j}(r)) \\
& \cong \left[ \varprojlim_n H^0(k,(({_{\ell^n}}A) ^{\boxtimes j})^D(1-r))\right] ^D \\
& \cong \left[ \varprojlim_n H^0(k,(({_{\ell^n}}A)^D) ^{\otimes j}(1-r))\right] ^D \\& \cong \left[\varprojlim_n H^0(k,({_{\ell^n}}A^t)^{\otimes j}(1-j-r))\right] ^D \\& \cong H^0(k, (T_{\ell}A^t)^{\otimes j} \otimes \mathbb{Z}_{\ell}(1-j-r))^D.
\end{align*}
According to theorem 4.2.2 of \cite{raynaud}, there exists a 1-motive $M=[Y \rightarrow G]$ over $k$ satisfying the following conditions:
\begin{itemize}
\item[(i)] there is a natural isomorphism $T_{\ell}A^t\cong T_{\ell} M$,
\item[(ii)] there is an exact sequence:
$$0 \rightarrow T \rightarrow G \rightarrow B \rightarrow 0$$
in which $T$ is a torus and $B$ is an abelian variety which has potentially good reduction.
\end{itemize}  
According to lemma \ref{motif}, one can deduce that the eigenvalues of a lifting $\phi$ of the geometric Frobenius acting on $T_{\ell}A^t$ have modulus $q^{-1}$, $q^{-1/2}$ or 1. Hence the eigenvalues of $\phi$ on $(T_{\ell}A^t)^{\otimes j} \otimes \mathbb{Z}_{\ell}(1-j-r)$ have modulus $q^{\nu} $ with $\nu \in \left( \frac{1}{2}\mathbb{Z}\right)  \cap \left[r-1, j+r-1 \right]$. Since $r\not\in \left[1-j,1 \right] $, none of these eigenvalues has modulus 1, and therefore $$H^0(k, (T_{\ell}A^t)^{\otimes j} \otimes \mathbb{Z}_{\ell}(1-j-r))=0.$$ The lemma follows.
\end{proof}

\begin{lemma}\label{abpadicp}
Let $k$ be a $p$-adic field. Let $A$ be an abelian variety over $k$. Let $j\geq 0$ and $r$ be integers such that $r\not\in \left[1-j,1 \right] $. Then $H^2(k,A\{p\}^{\boxtimes j} (r))=0$.
\end{lemma}

\begin{proof}[Proof]
Just as in the proof of the previous lemma, by Tate's local duality theorem, one has:
$$H^2(k,A\{p\}^{\boxtimes j} (r))  \cong H^0(k, (T_{p}A^t)^{\otimes j} \otimes \mathbb{Z}_{p}(1-j-r))^D.$$
We now write the Hodge-Tate decomposition (\cite{faltings}):
$$T_{p}A^t \otimes \mathbb{C}_p \cong H^1(A^t \times k^s,\mathbb{Q}_p(1))\otimes \mathbb{C}_p \cong H^1(A^t,\mathcal{O}_{A^t})\otimes \mathbb{C}_p(1) \oplus H^0(A^t,\Omega^1_{A^t/k})\otimes \mathbb{C}_p.$$
We deduce that there are $k$-vector spaces $V_0,...,V_j$ such that:
$$(T_{p}A^t)^{\otimes j} \otimes \mathbb{Z}_{p}(1-j-r) \otimes \mathbb{C}_p \cong \bigoplus_{i=0}^j \left( V_i \otimes \mathbb{C}_p(1-r-i)\right) .$$
Since $r\not\in \left[1-j,1 \right] $, none of the weights $1-r-i$ of the previous decomposition is 0. Hence $H^0(k, (T_{p}A^t)^{\otimes j} \otimes \mathbb{Z}_{p}(1-j-r))=0$, and this finishes the proof.
\end{proof}

\begin{theorem}\label{thpadique1}
Let $k$ be a $p$-adic field and let $\ell$ be a prime number. Let $M$ and $N$ be non-negative integers and consider a commutative, local, normal, henselian, excellent, $M$-dimensional ring $R$ with residue field $k$. Let $K$ be the function field of an $N$-dimensional integral variety over the fraction field of $R$. Then the group $$H^{M+N+2}(K,\mathbb{Q}_{\ell}/\mathbb{Z}_{\ell}(r))$$ vanishes for all $r\not\in \left[ 1,M+N+1 \right] $. 
\end{theorem}

\begin{proof}[Proof]
Observe that, by Hironaka's theorem, $K$ is the function field of an integral scheme $\mathcal{Y}$ such that $\mathcal{Y}^{sh}$ is integral and regular and such that there is a proper dominant morphism  $f:\mathcal{Y} \rightarrow \mathrm{Spec} \; R$. One then just has to combine corollary \ref{abs1bis} with lemmas \ref{abpadicell} and \ref{abpadicp}.
\end{proof}

\begin{remarque} If $K$ is the fraction field of $R$, the group $H^{M+2}(K,\mathbb{Q}_{\ell}/\mathbb{Z}_{\ell}(r))$ vanishes for all $r\not\in \left[ 1,M+1 \right] $.
\end{remarque}

\subsubsection{Number fields}

\begin{lemma}\label{abcdn}
Let $k$ be a number field. Let $A$ be an abelian variety over $k$ and let $\ell$ be a prime number which is assumed to be different from 2 if $k$ is not totally imaginary. Let $j\geq 0$ and $r$ be integers such that $r\not\in \left[1-j,1 \right] $. Then $H^2(k,A\{\ell\}^{\boxtimes j} (r))=0$.
\end{lemma}

\begin{proof}[Proof]
According to the Weil conjectures, if $v$ is a finite place of $K$ which does not divide $\ell$ and where $A$ has good reduction, all the eigenvalues of a lifting of the geometric Frobenius at $v$ acting on $T_{\ell}(A)$ have modulus $q^{-1/2}$. We deduce that all the eigenvalues of a lifting of the geometric Frobenius at $v$ acting on $T_{\ell}(A)^{\otimes j} \otimes \mathbb{Z}_{\ell}(r)$ have modulus $q^{-j/2-r}$. Since $r\not\in \left[1-j,1 \right] $, we have $q^{-j/2-r}\neq q^{-1}$. Theorem 1.5(a) of \cite{jannsen} hence implies that the morphism:
$$H^2(k,A\{\ell\}^{\boxtimes j} (r)) \rightarrow \bigoplus_{v} H^2(k_v,A\{\ell\}^{\boxtimes j} (r))$$
is injective. But according to lemmas \ref{abpadicell} and \ref{abpadicp}, since $r\not\in \left[1-j,1 \right]$, we have $H^2(k_v,A\{\ell\}^{\boxtimes j} (r))=0$ for each finite place $v$ of $k$. We deduce that the group $H^2(k,A\{\ell\}^{\boxtimes j} (r))$ vanishes.
\end{proof}

\begin{theorem}\label{thcdn1}
Let $k$ be a number field and let $\ell$ be a prime number which is assumed to be different from 2 if $k$ is not totally imaginary. Let $M$ and $N$ be non-negative integers and consider a commutative, local, normal, henselian, excellent, $M$-dimensional ring $R$ with residue field $k$. Let $K$ be the function field of an $N$-dimensional integral variety over the fraction field of $R$. Then the group $$H^{M+N+2}(K,\mathbb{Q}_{\ell}/\mathbb{Z}_{\ell}(r))$$ vanishes for all $r\not\in \left[ 1,M+N+1 \right] $. 
\end{theorem}

\begin{proof}[Proof]
Observe that, by Hironaka's theorem, $K$ is the function field of an integral scheme $\mathcal{Y}$ such that $\mathcal{Y}^{sh}$ is integral and regular and such that there is a proper dominant morphism  $f:\mathcal{Y} \rightarrow \mathrm{Spec} \; R$. One then just has to combine corollary \ref{abs1bis} with lemma \ref{abcdn}.
\end{proof}

\begin{remarque}
If $K$ is the fraction field of $R$, the group $H^{M+2}(K,\mathbb{Q}_{\ell}/\mathbb{Z}_{\ell}(r))$ vanishes for all $r\not\in \left[ 1,M+1 \right] $.
\end{remarque}

\section{Fraction fields of two-dimensional henselian local rings}\label{(())}

\hspace{3ex} In this section, we are interested in the function fields of two-dimensional henselian local rings. We will first improve the vanishing theorems of the previous section in this particular case, and we will then prove some Brauer-Hasse-Noether exact sequences.

\subsection{An abstract vanishing theorem}\label{ppp}

\hspace{3ex} Fix a perfect field $k$ and let $R$ be a commutative, local, geometrically integral, normal, henselian, excellent, 2-dimensional $k$-algebra with residue field $k$. By geometrically integral, we mean that $R\otimes_k k^s$ is integral. For instance, $R$ could be the completion or the henselization of the local ring at a closed point of a normal $k$-surface. Let $K$ be the fraction field of $R$ and let $\mathfrak{m}$ be the maximal ideal of $R$. Set $\mathcal{X} = \text{Spec} \; R$ and $X = \mathcal{X} \setminus \{\mathfrak{m}\}$. The scheme $\mathcal{X}$ may be singular in general, while $X$ is a (non-affine) Dedekind scheme.

\begin{theorem}\label{abs}
Let $d$ be a non-negative integer, $r$ any integer and $\ell$ a prime number different from the characteristic of $k$. We make the following three assumptions:
\begin{itemize}
\item[(H1)] The field $k$ has $\ell$-cohomological dimension $d$.
\item[(H2)] For each finite extension $k'$ of $k$, the group $H^d(k',\mathbb{Q}_{\ell}/\mathbb{Z}_{\ell}(r-2))$ vanishes.
\item[(H3)] For each finite extension $k'$ of $k$ and for each abelian variety $A$ over $k'$, the group $H^d(k',A(k^s)\{\ell\}(r-2))$ vanishes.
\end{itemize}
Then the group $H^{d+2}(K,\mathbb{Q}_{\ell}/\mathbb{Z}_{\ell}(r))$ vanishes.
\end{theorem}

\subsubsection{Resolution of singularities}\label{desingu}

\hspace{3ex} Consider a morphism of schemes $f: \tilde{\mathcal{X}} \rightarrow \mathcal{X} = \text{Spec} \; R$ satisfying the following assumptions:
\begin{itemize}
\item[$\bullet$] $\tilde{\mathcal{X}}$ is a regular, integral, 2-dimensional scheme and $f$ is projective;
\item[$\bullet$] $f: f^{-1}(X) \rightarrow X$ is an isomorphism;
\item[$\bullet$] $f^{-1}(\mathfrak{m})$ is a normal crossing divisor of $\mathcal{X}$ (in the sense of definition 9.1.6 of \cite{liu}).
\end{itemize}

Such a morphism exists according to \cite{lipman}. We then set $Y = f^{-1}(\mathfrak{m})$ and we endow $Y$ with the reduced structure. Thus $Y$ is a reduced $k$-curve which may not be irreducible, but all of its irreducible components are smooth. For $v \in \tilde{\mathcal{X}}^{(1)} \setminus X^{(1)} = Y^{(0)}$, let $Y_v$ be the smooth projective $k$-curve corresponding to $v$. We denote by $g_v$ the genus of this curve.\\

\hspace{3ex} Moreover, to the curve $Y$ we associate the following bipartite graph $\Gamma$:
\begin{itemize}
\item[$\bullet$] vertices: $V= V_1 \sqcup V_2$, where $V_1 = Y^{(0)}$ is the set of (generic points of) irreducible components of $Y$ and $V_2$ is the set of closed points of $Y$ which are the intersection of two irreducible components of $Y$;
\item[$\bullet$] edges: $E$ is the set of subsets with two elements $\{v_1,v_2\} $ of $V$ such that $v_1 \in V_1$, $v_2 \in V_2$ and $v_2\in Y_{v_1}$.
\end{itemize}
Denote by $c_{\Gamma}$ the first Betti number of $\Gamma$. \\

\hspace{3ex} In the sequel, when $S$ is a $k$-algebra or a $k$-scheme, $\underline{S}$ will stand for the extension of scalars of $S$ to $k^s$. Note that, in the same way that we have associated the graph $\Gamma$ to the curve $Y$, one can associate a bipartite graph $\underline{\Gamma}$ to the curve $\underline{Y}$: its set of vertices will be denoted $\underline{V}=\underline{V}_1 \sqcup \underline{V}_2$, and its set of edges $\underline{E}$. The notation $c_{\underline{\Gamma}}$ will then stand for the first Betti number of $\underline{\Gamma}$. For $\underline{v}\in \underline{V}_1$, we will denote by $\underline{Y}_{\underline{v}}$ the irreducible component of $\underline{Y}$ associated to $\underline{v}$ and by $g_{\underline{v}}$ the genus of $\underline{Y}_{\underline{v}}$. We set:
$$n_{\underline{X}} = \sum_{\underline{v} \in \underline{V}_1} g_{\underline{v}} + c_{\underline{\Gamma}}.$$

\subsubsection{The Brauer group of $\underline{K}$}

\hspace{3ex} In this paragraph, we are interested in the Brauer group of the field $\underline{K}$. Recall that this group has been computed in the article \cite{diego2}:

\begin{theorem} (Lemmas 1.4 and 1.6 and theorem 1.6 of \cite{diego2}) \label{BHN}\\
Let $\ell$ be a prime number different from the characteristic of $k$.
\begin{itemize}
\item[(i)] There is a natural exact sequence
\begin{equation}\label{brauer}
0 \rightarrow \Upsilon\{\ell\} \rightarrow \text{Br}(\underline{K})\{\ell\} \rightarrow \bigoplus_{\underline{v}\in \underline{X}^{(1)}} \text{Br}(\underline{K}_{\underline{v}})\{\ell\} \rightarrow \Lambda\{\ell\} \rightarrow 0,
\end{equation}
with:
\begin{gather*}
\Upsilon = \text{Ker}\left( \bigoplus_{\underline{v} \in \underline{V}_1} \text{Br} \; \underline{K}_{\underline{v}} \rightarrow \bigoplus_{\underline{w} \in \underline{\tilde{\mathcal{X}}}^{(2)}} \; \mathbb{Q}/\mathbb{Z}\right),\\
\Lambda = \text{Coker}\left( \bigoplus_{\underline{v} \in \underline{V}_1} \text{Br} \; \underline{K}_{\underline{v}} \rightarrow \bigoplus_{\underline{w} \in \underline{\tilde{\mathcal{X}}}^{(2)}} \; \mathbb{Q}/\mathbb{Z}\right).
\end{gather*}
\item[(ii)] There are isomorphisms of abelian groups $\Upsilon\{\ell\}\cong (\mathbb{Q}_{\ell}/\mathbb{Z}_{\ell})^{n_{\underline{X}}}$ and $\Lambda\{\ell\} \cong \mathbb{Q}_{\ell}/\mathbb{Z}_{\ell}$.
\end{itemize}
\end{theorem}

\begin{remarque}
In \cite{diego2}, this theorem is proved under the assumption that $\underline{Y}$ is a normal crossing divisor divisor of $\underline{\tilde{\mathcal{X}}}$ and this is not always the case. However, the proof only uses the fact that each irreducible component of $\underline{Y}$ is smooth, and this assumption is satisfied here.
\end{remarque}

\hspace{3ex} We aim at computing the group $\text{Br}(\underline{K})$ as a Galois module over $k$. For this purpose, we are going to compute $\Upsilon$ and $\Lambda$ as $\gal (k^s/k)$-modules, but before that, we introduce some notation:

\begin{notation}
For $\underline{v}\in \underline{\tilde{\mathcal{X}}}$:
\begin{itemize}
\item[$\bullet$] $v$ denotes the image of $\underline{v}$ in $\tilde{\mathcal{X}}$,
\item[$\bullet$]  $H_{\underline{v}}$ denotes the stabilizer of $\underline{v}$ under the action of $\gal (k^s/k)$,
\item[$\bullet$]  $F_{\underline{v}}$ stands for the field $ (k^s)^{H_{\underline{v}}}$.
\end{itemize}
When $\underline{v}\in \underline{Y}^{(1)}$, we also denote by $J_{\underline{v}}$ the Jacobian of the irreducible component of $Y_{v} \times_k F_{\underline{v}}$ associated to $\underline{v}$. Finally, we let $H_1(\underline{\Gamma},\mathbb{Z})$ be the first homology group of $\underline{\Gamma}$. Since the Galois group of $k$ acts on $\underline{\Gamma}$ and stabilizes $\underline{V}_1$ and $\underline{V}_2$, the group $H_1(\underline{\Gamma},\mathbb{Z})$ is naturally endowed with a structure of Galois module over $k$.
\end{notation}

\begin{lemma}
Let $\ell$ be a prime number different from the characteristic of $k$ et let $\underline{v}\in \underline{X}^{(1)}$. There is an isomorphism of Galois modules over $k$:
$$\bigoplus_{\sigma \in \gal (k^s/k)/H_{\underline{v}}} \text{Br}(\underline{K}_{\sigma (\underline{v})})\{\ell\} \cong I_{F_{\underline{v}}/k}\left(\mathbb{Q}_{\ell}/\mathbb{Z}_{\ell}(-1)\right).$$
\end{lemma}

\begin{proof}[Proof]
It suffices to observe that the residue morphisms induce canonical isomorphisms of $ H_{\underline{v}}$-modules:
\begin{align*}
\text{Br}(\underline{K}_{\underline{v}})\{\ell\}& \cong H^2(\underline{K}_{\underline{v}},\mathbb{Q}_{\ell}/\mathbb{Z}_{\ell}(1)) \cong H^1(k^s(\underline{v}),\mathbb{Q}_{\ell}/\mathbb{Z}_{\ell})\\& \cong H^0(k^s,\mathbb{Q}_{\ell}/\mathbb{Z}_{\ell}(-1)) \cong \mathbb{Q}_{\ell}/\mathbb{Z}_{\ell}(-1).
\end{align*}
\end{proof}

\begin{corollary}\label{lambda}
Let $\ell$ be a prime number different from the characteristic of $k$. There is an isomorphism of Galois modules over $k$:
$$\Lambda\{\ell\}\cong \mathbb{Q}_{\ell}/\mathbb{Z}_{\ell}(-1).$$
\end{corollary}

\begin{proof}[Proof]
Exact sequence (\ref{brauer}) gives a surjection of Galois modules over $k$:
$$\bigoplus_{\underline{v}\in \underline{X}^{(1)}} \text{Br}(\underline{K}_{\underline{v}})\{\ell\} \rightarrow \Lambda\{\ell\}$$
which is induced by the sum morphism. Since $\text{Br}(\underline{K}_{\underline{v}})\{\ell\} \cong \mathbb{Q}_{\ell}/\mathbb{Z}_{\ell}(-1)$ for each $\underline{v}\in \underline{X}^{(1)}$, we deduce that $\Lambda\{\ell\}\cong \mathbb{Q}_{\ell}/\mathbb{Z}_{\ell}(-1)$.
\end{proof}

Let's now study the group $\Upsilon$.

\begin{proposition}\label{upsilon}
Let $\ell$ be a prime number different from the characteristic of $k$. Let $|\underline{V}_1/\gal (k^s/k)|$ be a complete set of representatives of $\underline{V}_1/\gal (k^s/k)$. There is an exact sequence of Galois modules over $k$:
$$0 \rightarrow \bigoplus_{\underline{v} \in |\underline{V}_1/\gal (k^s/k)|} I_{F_{\underline{v}}/k}(J_{\underline{v}}(k^s)\{\ell\}(-1)) \rightarrow \Upsilon\{\ell\} \rightarrow \mathbb{Q}_{\ell}/\mathbb{Z}_{\ell}(-1)\otimes H_1(\underline{\Gamma},\mathbb{Z}) \rightarrow 0.$$
\end{proposition}

\begin{proof}[Proof]
We follow the computation in lemma 1.5 of \cite{diego2}, but here we have to be careful enough to control the action of the absolute Galois group of $k$. \\

For $\underline{v} \in \underline{V}_1$, consider the composed morphism:
$$\phi_{\underline{v}}: H^1(k^s(\underline{Y}_{\underline{v}}),\mathbb{Q}_{\ell}/\mathbb{Z}_{\ell}) \rightarrow \bigoplus_{\underline{w} \in \underline{Y}_{\underline{v}}^{(1)}} H^1(k^s(\underline{Y}_{\underline{v}})_{\underline{w}},\mathbb{Q}_{\ell}/\mathbb{Z}_{\ell}) \rightarrow  \bigoplus_{\underline{w} \in \underline{Y}_{\underline{v}}^{(1)}} \mathbb{Q}_{\ell}/\mathbb{Z}_{\ell}(-1),$$
where the first map is the restriction morphism and the second is induced by the residue morphisms. By summing the $\phi_{\underline{v}}$'s for $\underline{v} \in \underline{V}_1$, one gets a morphism of Galois modules:
$$\tilde{\phi}: \bigoplus_{\underline{v} \in \underline{V}_1} H^1(k^s(\underline{Y}_{\underline{v}}),\mathbb{Q}_{\ell}/\mathbb{Z}_{\ell}) \rightarrow \bigoplus_{\underline{v} \in \underline{V}_1} \bigoplus_{\underline{w} \in \underline{Y}_{\underline{v}}^{(1)}} \mathbb{Q}_{\ell}/\mathbb{Z}_{\ell}(-1).$$
Now define the morphism of Galois modules $\phi$ by the following diagram:\\ 
\centerline{\xymatrix{
\bigoplus_{\underline{v} \in \underline{V}_1} H^1(k^s(\underline{Y}_{\underline{v}}),\mathbb{Q}_{\ell}/\mathbb{Z}_{\ell}) \ar[r]^-{\tilde{\phi}}\ar[rd]^-{\phi} &\bigoplus_{\underline{v} \in \underline{V}_1} \bigoplus_{\underline{w} \in \underline{Y}_{\underline{v}}^{(1)}} \mathbb{Q}_{\ell}/\mathbb{Z}_{\ell}(-1)\ar[d]^{\Sigma}\\
& \bigoplus_{\underline{w} \in \underline{\tilde{\mathcal{X}}}^{(2)}} \mathbb{Q}_{\ell}/\mathbb{Z}_{\ell}(-1)
}}
 where the right vertical morphism is the sum morphism. As a Galois module, the group $\Upsilon\{\ell\}$ is the kernel of $ \phi$.\\
 
 Now note that each $\phi_{\underline{v}}$ can be inserted in the following exact sequence of $H_{\underline{v}}$-modules:
\begin{equation}\label{cbecomplexe}
\xymatrix{
H^1(k^s(\underline{Y}_{\underline{v}}),\mathbb{Q}_{\ell}/\mathbb{Z}_{\ell}) \ar[r]^-{\phi_{\underline{v}}} & \bigoplus_{\underline{w} \in \underline{Y}_{\underline{v}}^{(1)}} \mathbb{Q}_{\ell}/\mathbb{Z}_{\ell}(-1) \ar[r]^-{\Sigma} & \mathbb{Q}_{\ell}/\mathbb{Z}_{\ell}(-1) \ar[r] & 0.
}
\end{equation}
We therefore get a commutative diagram with exact lines and columns of Galois modules over $k$:\\
\centerline{\xymatrix{
 &  &  0 \ar[d] & 0 \ar[d] &\\
 &  &  \bigoplus_{\underline{v} \in \underline{V}_1} \text{Ker}(\phi_{\underline{v}})\ar[d] & \Upsilon\{\ell\} \ar[d]&\\
 & 0 \ar[r] &  \bigoplus_{\underline{v} \in \underline{V}_1} H^1(k^s(\underline{Y}_{\underline{v}}),\mathbb{Q}_{\ell}/\mathbb{Z}_{\ell})\ar[r]^{\cong}\ar[d]^{\bigoplus_{\underline{v} \in \underline{V}_1} \phi_{\underline{v}}} & \bigoplus_{\underline{v} \in \underline{V}_1} H^1(k^s(\underline{Y}_{\underline{v}}),\mathbb{Q}_{\ell}/\mathbb{Z}_{\ell})\ar[r]\ar[d]^{\phi} & 0\\
 0 \ar[r] & \Theta \ar[r] & \bigoplus_{\underline{v} \in \underline{V}_1} \text{Ker} \left(  \bigoplus_{\underline{w} \in \underline{Y}_{\underline{v}}^{(1)}} \mathbb{Q}_{\ell}/\mathbb{Z}_{\ell}(-1) \xrightarrow[]{\Sigma} \mathbb{Q}_{\ell}/\mathbb{Z}_{\ell}(-1)\right) \ar[r]^-{\Sigma} \ar[d]& \bigoplus_{w \in \underline{\tilde{\mathcal{X}}}^{(2)}} \mathbb{Q}_{\ell}/\mathbb{Z}_{\ell}(-1) &\\
  &   & 0 &  &
}}
where: 
$$\Theta = \text{Ker}\left( \Sigma: \bigoplus_{\underline{v} \in \underline{V}_1} \text{Ker} \left( \Sigma: \bigoplus_{\underline{w} \in \underline{Y}_{\underline{v}}^{(1)}} \mathbb{Q}_{\ell}/\mathbb{Z}_{\ell}(-1) \rightarrow \mathbb{Q}_{\ell}/\mathbb{Z}_{\ell}(-1)\right) \rightarrow \bigoplus_{\underline{w} \in \underline{\tilde{\mathcal{X}}}^{(2)}} \mathbb{Q}_{\ell}/\mathbb{Z}_{\ell}(-1) \right).$$
By the snake lemma, we get an exact sequence of Galois modules:
$$0 \rightarrow  \bigoplus_{\underline{v} \in \underline{V}_1} \text{Ker}(\phi_{\underline{v}}) \rightarrow \Upsilon\{\ell\} \rightarrow  \Theta \rightarrow 0.$$

By the proof of lemma 1.2 of \cite{diego2}, we have isomorphisms of Galois modules:
\begin{align*}
\Theta &=\text{Ker}\left( \Sigma: \bigoplus_{\underline{v} \in \underline{V}_1} \text{Ker} \left( \Sigma: \bigoplus_{\underline{w} \in \underline{Y}_{\underline{v}}^{(1)}} \mathbb{Q}_{\ell}/\mathbb{Z}_{\ell} \rightarrow \mathbb{Q}_{\ell}/\mathbb{Z}_{\ell}\right) \rightarrow \bigoplus_{\underline{w} \in \underline{\tilde{\mathcal{X}}}^{(2)}} \mathbb{Q}_{\ell}/\mathbb{Z}_{\ell} \right)(-1)\\
&=\mathbb{Q}_{\ell}/\mathbb{Z}_{\ell} (-1)\otimes H_1(\underline{\Gamma},\mathbb{Z}).
\end{align*}
 Besides, if $\underline{v} \in \underline{V}_1$, the $H_{\underline{v}}$-module $\text{Ker} \; \phi_{\underline{v}}$ is isomorphic to $ H^1(\underline{Y}_{\underline{v}}, \mathbb{Q}_{\ell}/\mathbb{Z}_{\ell})$. Since $\underline{Y}_{\underline{v}}$ is projective and since the group $H^1(\underline{Y}_{\underline{v}}, \mathbb{Q}_{\ell}/\mathbb{Z}_{\ell})$ can be identified with $(\text{Pic}\; \underline{Y}_{\underline{v}})\{\ell\}(-1)\cong J_{\underline{v}}(k^s)\{\ell\}(-1)$, we get:
$$ \bigoplus_{\underline{v} \in \underline{V}_1} \text{Ker}(\phi_{\underline{v}}) \cong \bigoplus_{\underline{v} \in \underline{V}_1} J_{\underline{v}}(k^s)\{\ell\}(-1) \cong \bigoplus_{\underline{v} \in |\underline{V}_1/\gal (k^s/k)|} I_{F_{\underline{v}}/k}(J_{\underline{v}}(k^s)\{\ell\}(-1)).$$
 
\end{proof}

\subsubsection{Proof of theorem \ref{abs}}

\hspace{3ex} We now prove theorem \ref{abs}. For this purpose, we write the Hochschild-Serre spectral sequence:
$$H^s(k,H^t(\underline{K},\mathbb{Q}_{\ell}/\mathbb{Z}_{\ell}(r))) \Rightarrow H^{s+t}(K,\mathbb{Q}_{\ell}/\mathbb{Z}_{\ell}(r))).$$
The group $H^s(k,H^t(\underline{K},\mathbb{Q}_{\ell}/\mathbb{Z}_{\ell}(r)))=0$ vanishes for $s\geq d+1$ or $t\geq 3$. Hence it suffices to prove that $H^d(k,H^2(\underline{K},\mathbb{Q}_{\ell}/\mathbb{Z}_{\ell}(r)))=0$. We see that:
\begin{align}
H^d(k,H^2(\underline{K},\mathbb{Q}_{\ell}/\mathbb{Z}_{\ell}(r)))&\cong H^d(k,H^2(\underline{K},\mathbb{Q}_{\ell}/\mathbb{Z}_{\ell}(1))(r-1))\\
&\cong H^d(k,\br(\underline{K})\{\ell\}(r-1)).
\end{align}
By theorem \ref{BHN}, corollary \ref{lambda} and proposition \ref{upsilon}, we have a dévissage of the Galois module $\br(\underline{K})(r-1)$:
\begin{gather}
0 \rightarrow \Upsilon\{\ell\}(r-1) \rightarrow \text{Br}(\underline{K})\{\ell\}(r-1) \rightarrow \text{Ker} \left( \bigoplus_{\underline{v}\in \underline{X}^{(1)}} \mathbb{Q}_{\ell}/\mathbb{Z}_{\ell}(r-2) \rightarrow \mathbb{Q}_{\ell}/\mathbb{Z}_{\ell}(r-2) \right) \rightarrow 0, \label{dev1}\\
0 \rightarrow \bigoplus_{\underline{v} \in |\underline{V}_1/\gal (k^s/k)|} I_{F_{\underline{v}}/k}(J_{\underline{v}}(k^s)\{\ell\}(r-2)) \rightarrow \Upsilon\{\ell\}(r-1) \rightarrow \mathbb{Q}_{\ell}/\mathbb{Z}_{\ell}(r-2)\otimes H_1(\underline{\Gamma},\mathbb{Z}) \rightarrow 0.\label{dev2}
\end{gather}
Since $r\neq 2$, assumptions (H2) and (H3) show that:
\begin{itemize}
\item[$\bullet$] $H^d(k,I_{F_{\underline{v}}/k}(J_{\underline{v}}(k^s)\{\ell\}(r-2)))=0$,
\item[$\bullet$] $H^d\left( k, \text{Ker} \left( \bigoplus_{\underline{v}\in \underline{X}^{(1)}} \mathbb{Q}_{\ell}/\mathbb{Z}_{\ell}(r-2) \rightarrow \mathbb{Q}_{\ell}/\mathbb{Z}_{\ell}(r-2) \right) \right)$ and $H^d(k, \mathbb{Q}_{\ell}/\mathbb{Z}_{\ell}(r-2)\otimes H_1(\underline{\Gamma},\mathbb{Z}))$ have finite exponent (by a restriction-corestriction argument).
\end{itemize}
We deduce that $H^d(k,\br(\underline{K})\{\ell\}(r-1))$ has finite exponent. Therefore $H^{d+2}(K,\mathbb{Q}_{\ell}/\mathbb{Z}_{\ell}(r))$ also has finite exponent. But this group is divisible since $K$ has $\ell$-cohomological dimension $d+2$ (assumption (H1)). Hence it vanishes.

\subsection{Applications}\label{app2}

\hspace{3ex} In this paragraph, we apply theorem \ref{abs} to the situations where $k$ is finite, $p$-adic or a totally imaginary number field. We keep the notations of section \ref{ppp}.

\begin{theorem}\label{thdim2}
Let $\ell$ be a prime number. 
\begin{itemize}
\item[(i)] If $k$ is finite and $\ell$ is different from the characteristic of $k$, then $H^3(K,\mathbb{Q}_{\ell}/\mathbb{Z}_{\ell}(r))$ vanishes for all $r\neq 2$.
\item[(ii)] If $k$ is $p$-adic, then $H^4(K,\mathbb{Q}_{\ell}/\mathbb{Z}_{\ell}(r))$ vanishes for all $r\neq 3$.
\item[(iii)] If $k$ is a totally imaginary number field or if $k$ is any number field and $\ell \neq 2$, then $H^4(K,\mathbb{Q}_{\ell}/\mathbb{Z}_{\ell}(r))$ vanishes for each $r\neq 3$.
\end{itemize}

\end{theorem}

\begin{proof}[Proof]
We check that the assumptions of theorem \ref{abs} hold:
\begin{itemize}
\item[(H1)] the field $k$ has $\ell$-cohomological dimension 1 if it is finite, 2 if it is $p$-adic or a totally imaginary number field.
\item[(H2-3)] This is a direct consequence of lemmas \ref{abfin}, \ref{abpadicell}, \ref{abpadicp} and \ref{abcdn}.
\end{itemize}
\end{proof}

\subsection{Brauer-Hasse-Noether exact sequences}\label{parbk}

\hspace{3ex} \emph{Keep all the notations of section \ref{ppp}}. In this paragraph, we are interested in the groups $H^{d+2}(K,\mathbb{Q}_{\ell}/\mathbb{Z}_{\ell}(r))$ when they do not vanish. We aim at proving some exact sequences which involve these groups and which play the role of the Brauer-Hasse-Noether exact sequence for the field $K$.\\

\hspace{3ex}  Recall that we have a dévissage of the Galois module $\br(\underline{K})(r-1)$ given by the exact sequences (\ref{dev1}) and (\ref{dev2}). Let's rewrite this dévissage in a different way.

\begin{proposition}\label{gr}
Let $\ell$ be a prime number different from the characteristic of $k$. There are exact sequences of Galois modules over $k$:
\begin{gather}
0 \rightarrow \bigoplus_{\underline{v} \in |\underline{V}_1/\gal (k^s/k)|} I_{F_{\underline{v}}/k}(J_{\underline{v}}(k^s)\{\ell\}(r-2)) \rightarrow \br(\underline{K})(r-1) \rightarrow  \Xi(r-2) \rightarrow 0,\label{spl1}\\
0 \rightarrow \mathbb{Q}_{\ell}/\mathbb{Z}_{\ell}\otimes H_1(\underline{\Gamma},\mathbb{Z}) \rightarrow \Xi \xrightarrow[]{\varphi} \text{Ker}\left( \bigoplus_{\underline{v}\in \underline{X}^{(1)}} \mathbb{Q}_{\ell}/\mathbb{Z}_{\ell} \rightarrow \mathbb{Q}_{\ell}/\mathbb{Z}_{\ell} \right) \rightarrow 0,\label{spl2}
\end{gather}
 where:
$$\Xi:=\text{Ker}\left( \bigoplus_{\underline{v}\in \underline{X}^{(1)}} \mathbb{Q}_{\ell}/\mathbb{Z}_{\ell} \oplus \bigoplus_{\underline{v}\in \underline{V}_1} \text{Ker}\left( \bigoplus_{w\in \underline{Y}_{\underline{v}}^{(1)}} \mathbb{Q}_{\ell}/\mathbb{Z}_{\ell} \rightarrow \mathbb{Q}_{\ell}/\mathbb{Z}_{\ell} \right)  \rightarrow \bigoplus_{\underline{w}\in \underline{Y}^{(1)}} \mathbb{Q}_{\ell}/\mathbb{Z}_{\ell} \right).$$
 Moreover, there exists a positive integer $M$ and a morphism of Galois modules $$\psi: \text{Ker}\left( \bigoplus_{\underline{v}\in \underline{X}^{(1)}} \mathbb{Q}_{\ell}/\mathbb{Z}_{\ell} \rightarrow \mathbb{Q}_{\ell}/\mathbb{Z}_{\ell} \right) \rightarrow \Xi$$  such that $\varphi \circ \psi = M\cdot \emph{Id}$. 
\end{proposition}

\begin{proof}[Proof]
Exact sequences (\ref{spl1}) and (\ref{spl2}) are just a rewriting of exact sequences (\ref{dev1}) and (\ref{dev2}). So we only have to prove the existence of $\psi$.\\

To do so, choose a subtree $\mathcal{T}$ of $\underline{\Gamma}$ which contains all vertices of $\underline{\Gamma}$. Denote by $\mathcal{T}_1,...,\mathcal{T}_m$ all the conjugates of $\mathcal{T}$ under the action of the absolute Galois group of $k$ on $\underline{\Gamma}$. According to the proof of lemma 1.1 of \cite{diego2}, for $i\in \{1,...,m\}$, if we consider a family $(a_{\underline{v}})_{\underline{v}\in \underline{V}} \in (\mathbb{Q}_{\ell}/\mathbb{Z}_{\ell})^{\underline{V}}$ such that:
$$\sum_{\underline{v}\in \underline{V}_1} a_{\underline{v}} = \sum_{\underline{v}\in \underline{V}_2} a_{\underline{v}},$$
then we can find a \emph{unique} family $(x_{\underline{e}})_{\underline{e}\in \underline{E}}\in (\mathbb{Q}_{\ell}/\mathbb{Z}_{\ell})^{\underline{E}}$ such that, for all $\underline{v}\in \underline{V}$:
$$\sum_{\underline{e}\in I(\underline{v})} x_{\underline{e}} = a_{\underline{v}}$$
and $x_{\underline{e}}=0$ if $\underline{e}$ is not an edge of $\mathcal{T}_i$. We will denote by $X_i((a_{\underline{v}})_{\underline{v}\in \underline{V}})$ the family $(x_{\underline{e}})_{\underline{e}\in \underline{E}}$. \\

Let's now consider $(\alpha_{\underline{v}})_{\underline{v}\in \underline{X}^{(1)}} \in \text{Ker}\left( \bigoplus_{v\in \underline{X}^{(1)}} \mathbb{Q}_{\ell}/\mathbb{Z}_{\ell} \rightarrow \mathbb{Q}_{\ell}/\mathbb{Z}_{\ell} \right)$. Denote by $(\beta_{\underline{w}})_{\underline{w}\in \underline{Y}^{(1)}} \in \bigoplus_{\underline{w}\in \underline{Y}^{(1)}} \mathbb{Q}_{\ell}/\mathbb{Z}_{\ell}$ the image of $(\alpha_{\underline{v}})$. For $\underline{v}_1\in \underline{V}_1$, set $a_{\underline{v}_1} =\sum_{\underline{w}\in \underline{Y}_{\underline{v}_1}^{(1)} \setminus \underline{V}_2} \beta_{\underline{w}}$ and for $\underline{v}_2\in \underline{V}_2$, set $a_{\underline{v}_2}=-\beta_{\underline{v}_2}$. We compute:
\begin{align*}
\sum_{\underline{v}\in \underline{V}_1} a_{\underline{v}} &= \sum_{\underline{v}\in \underline{V}_1} \sum_{\underline{w}\in \underline{Y}_{\underline{v}}^{(1)} \setminus \underline{V}_2} \beta_{\underline{w}}\\
& = \sum_{\underline{w}\in \underline{Y}^{(1)} } \beta_{\underline{w}} - \sum_{\underline{w}\in \underline{V}_2}  \beta_{\underline{w}}\\
&=\sum_{\underline{v}\in \underline{V}_2} a_{\underline{v}}.
\end{align*}
For $\underline{v}_1\in \underline{V}_1$, we consider the family $y_{\underline{v}_1}=(y_{\underline{v}_1,\underline{w}})_{\underline{w}\in \underline{Y}_{\underline{v}}^{(1)}}$ defined in the following way:
\begin{itemize}
\item[$\bullet$] $y_{\underline{v}_1,\underline{w}} = -m\beta_{\underline{w}}$ if $\underline{w}\not\in \underline{V}_2$;
\item[$\bullet$] $y_{\underline{v}_1,\underline{w}} = \sum_{i=1}^m X_i((a_{\underline{v}})_{\underline{v}\in \underline{V}})_{\{\underline{v}_1,\underline{w}\}}$ if $\underline{w}\in \underline{V}_2$.
\end{itemize}
For $\underline{v}_1 \in \underline{V}_1, \underline{v}_2 \in \underline{V}_2$ and $\underline{w}_0\in \underline{Y}^{(1)} \setminus \underline{V}_2$, we compute:
\begin{align*}
\sum_{\underline{w} \in \underline{Y}_{\underline{v}_1}^{(1)}} y_{\underline{v}_1,\underline{w}} & = -m\sum_{\underline{w} \in \underline{Y}_{\underline{v}_1}^{(1)} \setminus \underline{V}_2} \beta_{\underline{w}} + \sum_{\underline{w} \in \underline{Y}_{\underline{v}_1}\cap \underline{V}_2} \sum_{i=1}^m X_i((a_{\underline{v}})_{\underline{v}\in \underline{V}})_{\{\underline{v}_1,\underline{w}\}}\\
& = -m\sum_{\underline{w} \in \underline{Y}_{\underline{v}_1}^{(1)} \setminus \underline{V}_2} \beta_{\underline{w}} + \sum_{i=1}^m a_{\underline{v}_1}\\
&= m \left( -\sum_{\underline{w} \in \underline{Y}_{\underline{v}_1}^{(1)} \setminus \underline{V}_2} \beta_{\underline{w}} + a_{\underline{v}_1} \right)=0;\\
\sum_{\underline{v}\in \underline{V}_1 | \underline{v}_2 \in \underline{Y}_{\underline{v}}^{(1)}} y_{\underline{v},\underline{v}_2} & =  \sum_{\underline{v}\in \underline{V}_1| \underline{v}_2\in \underline{Y}_{\underline{v}}^{(1)}} \sum_{i=1}^m X_i((a_{\underline{v}})_{\underline{v}\in \underline{V}})_{\{\underline{v},\underline{v}_2\}}\\
&=  ma_{\underline{v}_2} = -m\beta_{\underline{v}_2};\\
\sum_{\underline{v}\in  \underline{V}_1 | \underline{w}_0 \in \underline{Y}_{\underline{v}}^{(1)}} y_{\underline{v},\underline{w}_0} &=-m\beta_{\underline{w}_0}.
\end{align*}
This allows us to define a morphism of groups:
\begin{align*}
\psi: \text{Ker}\left( \bigoplus_{\underline{v}\in \underline{X}^{(1)}} \mathbb{Q}_{\ell}/\mathbb{Z}_{\ell} \rightarrow \mathbb{Q}_{\ell}/\mathbb{Z}_{\ell} \right) &\rightarrow \Xi\\
(\alpha_{\underline{v}})_{\underline{v}\in \underline{X}^{(1)}} &\mapsto  \left(m(\alpha_{\underline{v}})_{\underline{v}\in \underline{X}^{(1)}},\left((y_{\underline{v}_1,\underline{w}})_{\underline{w}\in \underline{Y}_{\underline{v}_1}^{(1)}}\right)_{\underline{v}_1\in \underline{V}_1}\right).
\end{align*}
One can then easily check that $\psi$ is a morphism of Galois modules by observing that, if $\sigma \in \gal (k^s/k)$ sends some $\mathcal{T}_i$ on some $\mathcal{T}_j$, then:
$$\sigma \cdot \left( X_i((a_{\underline{v}})_{\underline{v}\in \underline{V}}) \right) = X_j(\sigma \cdot (a_{\underline{v}})_{\underline{v}\in \underline{V}}).$$
Of course, we have $\varphi \circ \psi = m\cdot \emph{Id}$.
\end{proof}

\begin{remarque}
According to the previous proof, $M$ can be taken to be the g.c.d. $m(\underline{\Gamma})$ of the orders of orbits of the set of spanning trees of $\underline{\Gamma}$ under the action of the Galois group of $k$. In particular, if $\ell$ does not divide $m(\underline{\Gamma})$, the sequence (\ref{spl2}) is split.
\end{remarque}

\begin{theorem}\label{BHN2}
Let $d$ be a non-negative integer, $r$ any integer and $\ell$ a prime number different from the characteristic of $k$. Assume that $k$ has $\ell$-cohomological dimension $d$ and that, for any finite extension $k'$ of $k$, the corestriction map $H^{d-1}(k',\mathbb{Q}_{\ell}/\mathbb{Z}_{\ell}(r-2)) \rightarrow H^{d-1}(k,\mathbb{Q}_{\ell}/\mathbb{Z}_{\ell}(r-2))$ is surjective. Then there are exact sequences:
\begin{gather}
\bigoplus_{\underline{v} \in |\underline{V}_1/\gal (k^s/k)|} H^{d}(F_{\underline{v}},J_{\underline{v}}\{\ell\}(r-2)) \rightarrow H^{d+2}(K,\mathbb{Q}_{\ell}/\mathbb{Z}_{\ell}(r)) \rightarrow H^{d}(k,\Xi(r-2)) \rightarrow 0,\label{cohom1}\\
\xymatrix{
0 \ar[r]& F \rightarrow H^{d}(k,\mathbb{Q}_{\ell}/\mathbb{Z}_{\ell}(r-2)\otimes H_1(\underline{\Gamma},\mathbb{Z}))  \ar[r]& H^{d}(k,\Xi(r-2))  \ar[d]\\
0& H^d(k,\mathbb{Q}_{\ell}/\mathbb{Z}_{\ell}(r-2)) \ar[l]& \bigoplus_{v\in X^{(1)}} H^{d+2}(K_v,\mathbb{Q}_{\ell}/\mathbb{Z}_{\ell}(r))\ar[l].
} \label{cohom2}
\end{gather} 
where $F$ is a group killed by $m(\underline{\Gamma})$.
\end{theorem}

\begin{proof}[Proof]
First observe that we have an exact sequence:
\small
\begin{equation}
\xymatrix{
\bigoplus_{\underline{v}\in |\underline{X}^{(1)}/\text{Gal}(k^s/k)|} H^{d-1}\left(F_{\underline{v}}, \mathbb{Q}_{\ell}/\mathbb{Z}_{\ell}(r-2) \right) \ar[rr] & &  H^{d-1}\left(k, \mathbb{Q}_{\ell}/\mathbb{Z}_{\ell}(r-2) \right)\ar[ld]\\
& \mathclap{H^d\left(k,\text{Ker}\left( \bigoplus_{\underline{v}\in \underline{X}^{(1)}} \mathbb{Q}_{\ell}/\mathbb{Z}_{\ell}(r-2) \rightarrow \mathbb{Q}_{\ell}/\mathbb{Z}_{\ell}(r-2) \right)\right)} \phantom{AAAAAAA} \ar[ld] & \\
\bigoplus_{\underline{v}\in |\underline{X}^{(1)}/\text{Gal}(k^s/k)|} H^d\left(F_{\underline{v}}, \mathbb{Q}_{\ell}/\mathbb{Z}_{\ell}(r-2) \right) \ar[r] & H^d\left(k, \mathbb{Q}_{\ell}/\mathbb{Z}_{\ell}(r-2) \right)\ar[r] & 0
}
\end{equation}
\normalsize
By assumption, for $\underline{v}\in |\underline{X}^{(1)}/\text{Gal}(k^s/k)|$, the corestriction morphism $H^{d-1}\left(F_{\underline{v}}, \mathbb{Q}_{\ell}/\mathbb{Z}_{\ell}(r-2) \right) \rightarrow H^{d-1}\left(k, \mathbb{Q}_{\ell}/\mathbb{Z}_{\ell}(r-2) \right)$ is surjective, so we get the following exact sequence:
\begin{equation}
\xymatrix{
& 0 \ar[r] & H^d\left(k,\text{Ker}\left( \bigoplus_{\underline{v}\in \underline{X}^{(1)}} \mathbb{Q}_{\ell}/\mathbb{Z}_{\ell}(r-2) \rightarrow \mathbb{Q}_{\ell}/\mathbb{Z}_{\ell}(r-2) \right)\right)  \ar[d]  \\
0 & H^d\left(k, \mathbb{Q}_{\ell}/\mathbb{Z}_{\ell}(r-2) \right)\ar[l] & \bigoplus_{\underline{v}\in |\underline{X}^{(1)}/\text{Gal}(k^s/k)|} H^d\left(F_{\underline{v}}, \mathbb{Q}_{\ell}/\mathbb{Z}_{\ell}(r-2) \right) \ar[l]
}
\end{equation}
Now note that, for $\underline{v}\in |\underline{X}^{(1)}/\text{Gal}(k^s/k)|$, the residue morphisms induce isomorphisms:
$$H^d\left(F_{\underline{v}}, \mathbb{Q}_{\ell}/\mathbb{Z}_{\ell}(r-2) \right) \cong H^{d+1}\left(k(v), \mathbb{Q}_{\ell}/\mathbb{Z}_{\ell}(r-1) \right)\cong H^{d+2}\left(K_v, \mathbb{Q}_{\ell}/\mathbb{Z}_{\ell}(r) \right).$$
Hence we have an exact sequence:
\begin{equation}
\xymatrix{
& 0 \ar[r] & H^d\left(k,\text{Ker}\left( \bigoplus_{\underline{v}\in \underline{X}^{(1)}} \mathbb{Q}_{\ell}/\mathbb{Z}_{\ell}(r-2) \rightarrow \mathbb{Q}_{\ell}/\mathbb{Z}_{\ell}(r-2) \right)\right)  \ar[d]  \\
0 & H^d\left(k, \mathbb{Q}_{\ell}/\mathbb{Z}_{\ell}(r-2) \right)\ar[l] & \bigoplus_{v \in X^{(1)}} H^{d+2}\left(K_v, \mathbb{Q}_{\ell}/\mathbb{Z}_{\ell}(r) \right) \ar[l]
}
\end{equation}
Exact sequences (\ref{cohom1}) and (\ref{cohom2}) are then obtained by writing the long exact cohomology sequences associated to the short exact sequences (\ref{spl1}) and (\ref{spl2}), by using the fact that $k$ has $\ell$-cohomological dimension $d$ and by observing that one has the following isomorphisms:
\begin{align*}
H^{d+2}(K,\mathbb{Q}_{\ell}/\mathbb{Z}_{\ell}(r)) & \cong H^d(k,H^2(\underline{K},\mathbb{Q}_{\ell}/\mathbb{Z}_{\ell}(r)))\\
&\cong H^d(k,H^2(\underline{K},\mathbb{Q}_{\ell}/\mathbb{Z}_{\ell}(1))(r-1))\\
&\cong H^d(k,\br(\underline{K})\{\ell\}(r-1)).
\end{align*}

Let's now prove that $F$ is killed by $m(\underline{\Gamma})$. Take the following notations:
\begin{gather*}
A:= \mathbb{Q}_{\ell}/\mathbb{Z}_{\ell}(r-2)\otimes H_1(\underline{\Gamma},\mathbb{Z});\\
B:=\Xi(r-2);\\
C:=\text{Ker}\left( \bigoplus_{\underline{v}\in \underline{X}^{(1)}} \mathbb{Q}_{\ell}/\mathbb{Z}_{\ell}(r-2) \rightarrow \mathbb{Q}_{\ell}/\mathbb{Z}_{\ell}(r-2) \right).
\end{gather*}
According to proposition \ref{gr}, we have a commutative diagram with an exact line:
\begin{equation}
\xymatrix{
H^{d-1}(k,B)  \ar[r] & H^{d-1}(k,C)  \ar[r] & H^{d}(k,A)  \ar[r] & H^{d}(k,B)  \ar[r] & H^{d}(k,C)  \ar[r] & 0\\
H^{d-1}(k,C)\ar[u] \ar[ur]_{\cdot m(\underline{\Gamma})} &&&&&
}
\end{equation}
A simple diagram chase shows that  $F=\text{Ker}\left(H^{d}(k,A)  \rightarrow H^{d}(k,B) \right)$ is killed $m(\underline{\Gamma})$.
\end{proof}

\begin{corollary}\label{corobhn}
Let $d$ be a non-negative integer, $r$ any integer and $\ell$ a prime number different from the characteristic of $k$. Assume that $k$ has $\ell$-cohomological dimension $d$ and that, for any finite extension $k'$ of $k$, the corestriction map $H^{d-1}(k',\mathbb{Q}_{\ell}/\mathbb{Z}_{\ell}(r-2)) \rightarrow H^{d-1}(k,\mathbb{Q}_{\ell}/\mathbb{Z}_{\ell}(r-2))$ is surjective. There is an exact sequence:
$$H^{d+2}(K,\mathbb{Q}_{\ell}/\mathbb{Z}_{\ell}(r)) \rightarrow \bigoplus_{v\in X^{(1)}} H^{d+2}(K_v,\mathbb{Q}_{\ell}/\mathbb{Z}_{\ell}(r)) \rightarrow H^d(k,\mathbb{Q}_{\ell}/\mathbb{Z}_{\ell}(r-2)) \rightarrow 0.$$
Moreover, if $A$ is the kernel of the map $H^{d+2}(K,\mathbb{Q}_{\ell}/\mathbb{Z}_{\ell}(r)) \rightarrow \bigoplus_{v\in X^{(1)}} H^{d+2}(K_v,\mathbb{Q}_{\ell}/\mathbb{Z}_{\ell}(r))$, then  we have an exact sequence:
$$\bigoplus_{\underline{v} \in |\underline{V}_1/\gal (k^s/k)|} H^{d}(F_{\underline{v}},J_{\underline{v}}\{\ell\}(r-2)) \rightarrow A \rightarrow H^{d}(k,\mathbb{Q}_{\ell}/\mathbb{Z}_{\ell}(r-2)\otimes H_1(\underline{\Gamma},\mathbb{Z}))/F \rightarrow 0,$$
where $F$ is a subgroup of $ H^{d}(k,\mathbb{Q}_{\ell}/\mathbb{Z}_{\ell}(r-2)\otimes H_1(\underline{\Gamma},\mathbb{Z}))$ which is killed by $m(\underline{\Gamma})$. In particular, $A$ is divisible.
\end{corollary}

\begin{remarque}~
\begin{itemize}
\item[(i)] This corollary is to be understood as a Brauer-Hasse-Noether exact sequence for the field $K$.
\item[(ii)] The assumption that, for any finite extension $k'$ of $k$, the corestriction map $$H^{d-1}(k',\mathbb{Q}_{\ell}/\mathbb{Z}_{\ell}(r-2)) \rightarrow H^{d-1}(k,\mathbb{Q}_{\ell}/\mathbb{Z}_{\ell}(r-2))$$ is surjective is automatically satisfied when $r=d+1$.
\end{itemize}
\end{remarque}

\begin{proof}[Proof]
This is just a rewriting of theorem \ref{BHN2}, provided that one notes that the groups $\bigoplus_{\underline{v} \in \underline{V}_1/\gal (k^s/k)} H^{d}(F_{\underline{v}},J_{\underline{v}}\{\ell\}(r-2))$ and $H^{d}(k,\mathbb{Q}_{\ell}/\mathbb{Z}_{\ell}(r-2)\otimes H_1(\underline{\Gamma},\mathbb{Z}))$ are divisible (because $k$ has cohomological dimension $d$).
\end{proof}

We are now going to apply the previous corollary to the cases where $k$ is a finite field, a $p$-adic field or a totally imaginary number field.

\subsection*{a) Finite base field}

\hspace{3ex} In the case where $k$ is finite, we recover a result of Saito (see theorem 5.2 of \cite{saito}):

\begin{corollary}\label{bhnfin}
Assume that $k$ is finite. Then we have an exact sequence:
\begin{equation*}
\xymatrix{
0\ar[r] & F \ar[r] & H^{1}(k,\mathbb{Q}_{\ell}/\mathbb{Z}_{\ell}\otimes H_1(\underline{\Gamma},\mathbb{Z})) \ar[r] & H^{3}(K,\mathbb{Q}_{\ell}/\mathbb{Z}_{\ell}(2)) \ar[d]\\
&0 & H^1(k,\mathbb{Q}_{\ell}/\mathbb{Z}_{\ell})\ar[l] &\bigoplus_{v\in X^{(1)}} H^{3}(K_v,\mathbb{Q}_{\ell}/\mathbb{Z}_{\ell}(2)),\ar[l]
}
\end{equation*}
where $F$ is a group killed by $m(\underline{\Gamma})$. The kernel of the map $$H^{3}(K,\mathbb{Q}_{\ell}/\mathbb{Z}_{\ell}(2))\rightarrow \bigoplus_{v\in X^{(1)}} H^{3}(K_v,\mathbb{Q}_{\ell}/\mathbb{Z}_{\ell}(2))$$ is isomorphic to $(\mathbb{Q}_{\ell}/\mathbb{Z}_{\ell})^{\rho}$ with:
\begin{equation}\label{rang}
\rho=\mathrm{rank}\left( H_1(\underline{\Gamma},\mathbb{Z})^{\gal(k^s/k)} \right) .
\end{equation}
\end{corollary}

\begin{proof}[Proof]
Observe that for any finite extension $k'$ of $k$, the corestriction map $H^{0}(k',\mathbb{Q}_{\ell}/\mathbb{Z}_{\ell}) \rightarrow H^{0}(k,\mathbb{Q}_{\ell}/\mathbb{Z}_{\ell})$ is surjective. Hence the first part of the statement follows from corollary \ref{corobhn} and from the vanishing of the group $H^{1}(F_{\underline{v}},J_{\underline{v}}\{\ell\})$ for each $\underline{v} \in |\underline{V}_1/\gal (k^s/k)|$ (lemma \ref{abfin}). For the second part of the statement, we have by duality over $k$ (example I.1.10 of \cite{milne}):
$$H^{1}(k,\mathbb{Q}_{\ell}/\mathbb{Z}_{\ell}\otimes H_1(\underline{\Gamma},\mathbb{Z})) \cong \left( H^0(k, T_{\ell}(H_1(\underline{\Gamma},\mathbb{Z})^D)) \right)^D.$$
This group is isomorphic to $(\mathbb{Q}_{\ell}/\mathbb{Z}_{\ell})^{\rho}$ with:
$$ \rho = \mathrm{rank}\left( H^0(k, T_{\ell}(H_1(\underline{\Gamma},\mathbb{Z})^D)) \right).$$
Hence the corollary follows from the next lemma.
\end{proof}

\begin{lemma}\label{Ono}
Let $E$ be a field and let $M$ be a $\text{Gal}(E^s/E)$-module which is free of finite type as an abelian group. Then:
$$\mathrm{rank}\left( H^0(E,T_{\ell}(M^D)) \right)  = \mathrm{rank} \left( M^{\gal(E^s/E)} \right) .$$
\end{lemma}

\begin{proof}[Proof]
The result is immediate for $M$ a permutation module. By Ono's lemma, one can find a positive integer $m$ and an exact sequence of Galois modules:
$$0 \rightarrow M^m \times R_0 \rightarrow R_1 \rightarrow F \rightarrow 0$$
in which $R_0$ and $R_1$ are permutation modules and $F$ is finite. One then easily checks that there is an exact sequence of Galois modules:
$$0 \rightarrow T_{\ell}(R_1^D) \rightarrow T_{\ell}(M^D)^m \times T_{\ell}(R_0^D) \rightarrow F' \rightarrow 0$$
with $F'$ finite. It follows that:
\begin{align*}
\text{rank}\left( H^0(E,T_{\ell}(M^D)) \right)  &= \frac{\text{rank}\left( H^0(E,T_{\ell}(R_1^D)) \right)-\text{rank}\left( H^0(E,T_{\ell}(R_0^D)) \right)}{m}\\
&= \frac{\text{rank}\left( R_1^{\gal(E^s/E)} \right)-\text{rank}\left( R_0^{\gal(E^s/E)} \right)}{m}\\
&= \text{rank} \left( M^{\gal(E^s/E)} \right).
\end{align*}
\end{proof}

\subsection*{b) $p$-adic base field}

\hspace{3ex} When $k$ is a $p$-adic field, corollary \ref{corobhn} implies the following result:

\begin{corollary}\label{padicbhn}
\begin{itemize}
\item[(i)] Assume that $k$ is a $p$-adic field. Then we have an exact sequence:
$$H^{4}(K,\mathbb{Q}_{\ell}/\mathbb{Z}_{\ell}(3)) \rightarrow \bigoplus_{v\in X^{(1)}} H^{4}(K_v,\mathbb{Q}_{\ell}/\mathbb{Z}_{\ell}(3)) \rightarrow (\text{Br}\; k)\{\ell\} \rightarrow 0.$$
\item[(ii)] Assume further that  $\ell \neq p$.
Then the group $ \text{Ker}\left( H^{4}(K,\mathbb{Q}_{\ell}/\mathbb{Z}_{\ell}(3)) \rightarrow \bigoplus_{v\in X^{(1)}} H^{4}(K_v,\mathbb{Q}_{\ell}/\mathbb{Z}_{\ell}(3)) \right)$ is isomorphic to $(\mathbb{Q}_{\ell}/\mathbb{Z}_{\ell})^{\rho}$ with:
$$ 0 \leq \rho-\mathrm{rank}\left( H_1(\underline{\Gamma},\mathbb{Z})^{\gal(k^s/k)} \right) \leq \sum_{\underline{v} \in |\underline{V}_1/\gal (k^s/k)|} \rho_{\overline{v}}.$$
where $\rho_{\overline{v}}$ is the rank of the maximal torus of the special fiber of the Néron model of $J_{\overline{v}}$. Moreover, if $\rho_{\overline{v}}=0$ for each $\underline{v} \in |\underline{V}_1/\gal (k^s/k)|$, then we have an exact sequence:
\begin{equation*}
\xymatrix{
0\ar[r] & F \ar[r] & H^{2}(k,\mathbb{Q}_{\ell}/\mathbb{Z}_{\ell}(1)\otimes H_1(\underline{\Gamma},\mathbb{Z})) \ar[r] & H^{4}(K,\mathbb{Q}_{\ell}/\mathbb{Z}_{\ell}(3)) \ar[d]\\
&0 & (\text{Br}\; k)\{\ell\}\ar[l] &\bigoplus_{v\in X^{(1)}} H^{4}(K_v,\mathbb{Q}_{\ell}/\mathbb{Z}_{\ell}(3)),\ar[l]
}
\end{equation*}
where $F$ is a group killed by $m(\underline{\Gamma})$.
\item[(iii)] Assume that  $\ell = p$. Then the group $ \text{Ker}\left( H^{4}(K,\mathbb{Q}_p/\mathbb{Z}_p(3)) \rightarrow \bigoplus_{v\in X^{(1)}} H^{4}(K_v,\mathbb{Q}_p/\mathbb{Z}_p(3)) \right)$ is isomorphic to $(\mathbb{Q}_p/\mathbb{Z}_p)^{\rho}$ with:
$$ 0 \leq \rho-\mathrm{rank}\left( H_1(\underline{\Gamma},\mathbb{Z})^{\gal(k^s/k)} \right) \leq \sum_{\underline{v} \in |\underline{V}_1/\gal (k^s/k)|} g_{\underline{v}}.$$
Moreover, if $g_{\underline{v}}=0$ for each $\underline{v} \in |\underline{V}_1/\gal (k^s/k)|$, then we have an exact sequence:
\begin{equation*}
\xymatrix{
0\ar[r] & F \ar[r] & H^{2}(k,\mathbb{Q}_p/\mathbb{Z}_p(1)\otimes H_1(\underline{\Gamma},\mathbb{Z})) \ar[r] & H^{4}(K,\mathbb{Q}_p/\mathbb{Z}_p(3)) \ar[d]\\
&0 &(\text{Br}\; k)\{p\}\ar[l] &\bigoplus_{v\in X^{(1)}} H^{4}(K_v,\mathbb{Q}_p/\mathbb{Z}_p(3)),\ar[l]
}
\end{equation*}
where $F$ is a group killed by $m(\underline{\Gamma})$.
\end{itemize}

\end{corollary}

\begin{proof}[Proof]
Observe that, for any finite extension $k'$ of $k$, the corestriction morphism $$H^{1}(k',\mathbb{Q}_{\ell}/\mathbb{Z}_{\ell}(1)) \rightarrow H^{1}(k,\mathbb{Q}_{\ell}/\mathbb{Z}_{\ell}(1))$$ is surjective. Hence, corollary \ref{padicbhn} follows from corollary \ref{corobhn}, provided that we check that, for each $\underline{v} \in |\underline{V}_1/\gal (k^s/k)|$, 
\begin{equation} \label{cond1}
H^{2}(F_{\underline{v}},J_{\underline{v}}\{\ell\}(1)) \cong (\mathbb{Q}_{\ell}/\mathbb{Z}_{\ell})^{\tau_{\overline{v}}}
\end{equation}
 for some $\tau_{\overline{v}}$ satisfying the inequality
$$\tau_{\underline{v}}\leq \begin{cases} \rho_{\underline{v}}, & \mbox{if } \ell\neq p \\ g_{\underline{v}}, & \mbox{if } \ell=p \end{cases}$$  and that 
\begin{equation}\label{cond2}
\mathrm{rank} \left(H^{2}(k,\mathbb{Q}_{\ell}/\mathbb{Z}_{\ell}(1)\otimes H_1(\underline{\Gamma},\mathbb{Z}))\right) = \mathrm{rank} \left(H_1(\underline{\Gamma},\mathbb{Z})^{\gal(k^s/k)} \right).
\end{equation}
Equality (\ref{cond2}) can be proved in the same way as equality (\ref{rang}) if one uses Tate's local duality over the $p$-adic field $k$. As for the isomorphism (\ref{cond1}), one can use the following lemma.
\end{proof}

\begin{lemma}\label{qsd}
Let $A$ be an abelian variety over a $p$-adic field $k$ with residue field $\kappa$. Let $\ell$ be a prime number. Let $\mathcal{A}$ be the Néron model of $A$ and let $T$ be the maximal torus of the special fiber of $\mathcal{A}$. Let $\rho_T$ be the rank of $T$.
\begin{itemize}
\item[(i)] If $\ell \neq p$, then $H^2(k,A\{\ell\}(1))\cong (\mathbb{Q}_{\ell}/\mathbb{Z}_{\ell})^{\rho}$ with $\rho = \rho_T$.
\item[(ii)] If $\ell = p$, then $H^2(k,A\{\ell\}(1))\cong (\mathbb{Q}_{\ell}/\mathbb{Z}_{\ell})^{\rho}$ with $\rho \leq \dim \; A$.
\end{itemize}
\end{lemma}

\begin{proof}[Proof]
First observe that, by Tate's duality theorem:
$$H^2(k,A\{\ell\}(1))\cong H^0(k,T_{\ell}A^t \otimes \mathbb{Z}_{\ell}(-1))^D.$$
Since $A$ and $A^t$ are isogenous, there is a (non-canonical) isomorphism:
$$H^0(k,T_{\ell}A^t \otimes \mathbb{Z}_{\ell}(-1))\cong H^0(k,T_{\ell}A \otimes \mathbb{Z}_{\ell}(-1)).$$
Hence we have to prove that $H^0(k,T_{\ell}A \otimes \mathbb{Z}_{\ell}(-1))$ is a free $\mathbb{Z}_{\ell}$-module of rank $\rho_T$ if $\ell\neq p$ and of rank $\leq \dim \; A$ if $\ell=p$.
\begin{itemize}
\item[(i)] We first assume that $\ell \neq p$. We then have an isomorphism of $\gal (k^{\mathrm{unr}}/k)$-modules:
$$(T_{\ell}A)^{\gal (k^s/k^{\mathrm{unr}})} \cong \varprojlim_n {_{\ell^n}}A(k^{\mathrm{unr}}) \cong \varprojlim_n {_{\ell^n}}\mathcal{A}(\mathcal{O}_{k^{\mathrm{unr}}}) \cong T_{\ell}A_0$$
where $A_0$ is the special fiber of $\mathcal{A}$. Hence:
$$H^0(k,T_{\ell}A \otimes \mathbb{Z}_{\ell}(-1)) \cong H^0(\kappa,T_{\ell}A_0 \otimes \mathbb{Z}_{\ell}(-1)).$$
 Let now $A_0^0$ be the connected component of the unit in $A_0$. We then have exact sequences:
\begin{gather}\label{aze1}
0 \rightarrow A_0^0 \rightarrow A_0 \rightarrow F \rightarrow 0\\
\label{aze2}
0\rightarrow U\times T \rightarrow A_0^0 \rightarrow B \rightarrow 0,
\end{gather}
for some finite étale group scheme $F$, some unipotent group $U$ and some abelian variety $B$. By exact sequence (\ref{aze1}), we get an isomorphism of Galois modules:
$$T_{\ell}A_0 \cong T_{\ell}A_0^0,$$
so that:
$$H^0(\kappa,T_{\ell}A_0 \otimes \mathbb{Z}_{\ell}(-1)) \cong H^0(\kappa,T_{\ell}A_0^0 \otimes \mathbb{Z}_{\ell}(-1)).$$
Moreover, by exact sequence (\ref{aze2}), we have an exact sequence of Galois modules:
$$0 \rightarrow T_{\ell}T \rightarrow T_{\ell}A_0^0 \rightarrow T_{\ell}B.$$
We therefore have an exact sequence:
$$0 \rightarrow H^0(\kappa,T_{\ell}T \otimes \mathbb{Z}_{\ell}(-1)) \rightarrow H^0(\kappa,T_{\ell}A_0^0 \otimes \mathbb{Z}_{\ell}(-1)) \rightarrow H^0(\kappa,T_{\ell}B \otimes \mathbb{Z}_{\ell}(-1)).$$
But $H^0(\kappa,T_{\ell}B \otimes \mathbb{Z}_{\ell}(-1))=0$ by the Weil conjectures. So we get:
$$H^0(\kappa,T_{\ell}T \otimes \mathbb{Z}_{\ell}(-1)) \cong H^0(\kappa,T_{\ell}A_0^0 \otimes \mathbb{Z}_{\ell}(-1)).$$
Hence, to finish the proof, we have to show that $H^0(\kappa,T_{\ell}T \otimes \mathbb{Z}_{\ell}(-1))$ is a free $\mathbb{Z}_{\ell}$-module of rank $\rho_T$. This is obvious when $T$ is a quasi-trivial torus. Now, by Ono's lemma, we have an exact sequence:
$$0 \rightarrow F_0 \rightarrow R_0 \rightarrow T^n \times R_1 \rightarrow 0$$
where $F_0$ is a finite étale group scheme, $R_0$ and $R_1$ are quasitrivial tori and $n$ is some positive integer. Hence we have an exact sequence:
\begin{equation}\label{ono}
0 \rightarrow H^0(\kappa,T_{\ell}R_0 \otimes \mathbb{Z}_{\ell}(-1)) \rightarrow H^0(\kappa,T_{\ell}T \otimes \mathbb{Z}_{\ell}(-1))^n \times H^0(\kappa,T_{\ell}R_1 \otimes \mathbb{Z}_{\ell}(-1)) \rightarrow F_1 \rightarrow 0
\end{equation}
for some finite group $F_1$. Since $H^0(\kappa,T_{\ell}R_0 \otimes \mathbb{Z}_{\ell}(-1))$ and $H^0(\kappa,T_{\ell}R_1 \otimes \mathbb{Z}_{\ell}(-1))$ are free $\mathbb{Z}_{\ell}$-modules, so is $H^0(\kappa,T_{\ell}T \otimes \mathbb{Z}_{\ell}(-1))$. Moreover, by studying the ranks of the $\mathbb{Z}_{\ell}$-modules that appear in exact sequence (\ref{ono}), we see that:
\begin{align*}
\text{rank} \left( H^0(\kappa,T_{\ell}T \otimes \mathbb{Z}_{\ell}(-1)) \right) &= \frac{\text{rank} \left( H^0(\kappa,T_{\ell}R_0 \otimes \mathbb{Z}_{\ell}(-1)) \right) - \text{rank} \left( H^0(\kappa,T_{\ell}R_1 \otimes \mathbb{Z}_{\ell}(-1)) \right)}{n}\\& = \frac{\rho_{R_0}-\rho_{R_1}}{n} = \rho_T,
\end{align*}
as wished.
\item[(ii)] We now assume that $\ell=p$. According to the Hodge-Tate decomposition (\cite{faltings}), we have an isomorphism of Galois modules:
$$T_{p}A \otimes \mathbb{C}_{p}(-1) \cong H^1(A,\mathcal{O}_{A})\otimes \mathbb{C}_p \oplus H^0(A,\Omega^1_{A/k})\otimes \mathbb{C}_p(-1).$$
Since $\dim \; H^1(A,\mathcal{O}_{A}) = \dim \; A$, we deduce that $H^0(k,T_{p}A \otimes \mathbb{Z}_{p}(-1))$ is a $\mathbb{Z}_{\ell}$-module of rank at most $\dim \; A$.
\end{itemize}
\end{proof}

\subsection*{c) Number fields}

\hspace{3ex} The case when $k$ is a number field is summarized in the following corollary:

\begin{corollary}\label{bhncdn}
Assume that $k$ is a totally imaginary number field or that $k$ is any number field and that $\ell \neq 2$. We have an exact sequence: 
$$H^{4}(K,\mathbb{Q}_{\ell}/\mathbb{Z}_{\ell}(3)) \rightarrow \bigoplus_{v\in X^{(1)}} H^{4}(K_v,\mathbb{Q}_{\ell}/\mathbb{Z}_{\ell}(3)) \rightarrow (\text{Br}\; k)\{\ell\} \rightarrow 0.$$
Moreover, the group $\mathrm{Ker}\left( H^{4}(K,\mathbb{Q}_{\ell}/\mathbb{Z}_{\ell}(3)) \rightarrow \bigoplus_{v\in X^{(1)}} H^{4}(K_v,\mathbb{Q}_{\ell}/\mathbb{Z}_{\ell}(3)) \right)$ is divisible.
\end{corollary}

\begin{proof}[Proof]
Observe that, for any finite extension $k'$ of $k$, the corestriction morphism $$H^{1}(k',\mathbb{Q}_{\ell}/\mathbb{Z}_{\ell}(1)) \rightarrow H^{1}(k,\mathbb{Q}_{\ell}/\mathbb{Z}_{\ell}(1))$$ is surjective. Hence everything follows immediately from corollary \ref{corobhn}.
\end{proof}

We finish this article with a theorem which generalizes Jannsen's exact sequence (\ref{sejan}) to the field $K$. To do so, we first need to introduce some notation:

\begin{notation}
Assume that $k$ is a number field. Let $\Omega_k$ be the set of places of $k$ and let $S$ be a finite subset of $\Omega_k$. For every Galois module $M$ over $k$, we define the groups:
\begin{gather*}
\ch^1_S(k,M):= \text{Coker}\left( H^1(k,M) \rightarrow \bigoplus_{\pi\in S} H^1(k_{\pi},M) \right), \\
\ch^2(k,M):= \text{Coker}\left( H^2(k,M) \rightarrow \bigoplus_{\pi\in \Omega_k} H^2(k_{\pi},M) \right), \\
\Sha^2(k,M):= \text{Ker}\left( H^2(k,M) \rightarrow \prod_{\pi\in \Omega_k} H^2(k_{\pi},M) \right) . 
\end{gather*}
\end{notation}

\begin{theorem}\label{bhncdnbis}
Assume that $k$ is a number field and let $\Omega_k$ be the set of places of $k$. For $\pi \in \Omega_k$, denote by $k_{\pi}^h$ the henselization of $k$ at $\pi$ and by $K_{\pi}^h$ the field $K\otimes_k k_{\pi}^h$. Let $A$ be the $\text{Gal}(k^s/k)$-module $H_1(\underline{\Gamma},\mathbb{Z})\otimes \mathbb{Q}_{\ell}/\mathbb{Z}_{\ell}(1)$. Then there exists a natural complex $\mathcal{C}_K$:
\small
\begin{equation}
\xymatrix@R=2mm@C=5mm{
& \text{(degree 1)} & \text{(degree 2)} & \text{(degree 3)} & \text{(degree 4)} & \\
0 \ar[r] &  H^4(K,\mathbb{Q}_{\ell}/\mathbb{Z}_{\ell}(3)) \ar[r] & \bigoplus\limits_{\pi\in \Omega_k} H^4(K_{\pi}^h,\mathbb{Q}_{\ell}/\mathbb{Z}_{\ell}(3)) \ar[r] & \bigoplus\limits_{v\in X^{(1)}} \mathbb{Q}_{\ell}/\mathbb{Z}_{\ell} \ar[r] & \mathbb{Q}_{\ell}/\mathbb{Z}_{\ell} \ar[r] & 0
}
\end{equation}
\normalsize
such that:
\begin{itemize}
\item[(i)] $H^3(\mathcal{C}_K) = H^4(\mathcal{C}_K)=0$,
\item[(ii)] the group $H^2(\mathcal{C}_K)$ is the quotient of $\ch^2(k,A)$ by a subgroup killed by $m(\underline{\Gamma})$ and it is (non-canonically) isomorphic to $(\mathbb{Q}_{\ell}/\mathbb{Z}_{\ell})^{\rho}$ where $\rho = \text{rank}\left(H_1(\underline{\Gamma},\mathbb{Z})^{\text{Gal}(k^s/k)}\right)$,
\item[(iii)] the group $H^1(\mathcal{C}_K)$ has finite exponent.
\end{itemize}
If moreover $\ell$ does not divide $m(\underline{\Gamma})$, then there exist a finite subset $S$ of $\Omega_k$ and a natural exact sequence:
$$\ch^1_S(k,A) \rightarrow H^1(\mathcal{C}_K) \rightarrow \Sha^2(k,A) \rightarrow 0.$$
\end{theorem}

\begin{lemma} \label{lemfin}
Keep the notations of theorem \ref{bhncdnbis} and take the following notations:
\begin{gather*}
B:=\Xi(1);\\
C:=\text{Ker}\left( \bigoplus_{\underline{v}\in \underline{X}^{(1)}} \mathbb{Q}_{\ell}/\mathbb{Z}_{\ell}(1) \rightarrow \mathbb{Q}_{\ell}/\mathbb{Z}_{\ell}(1) \right).
\end{gather*} 
\begin{itemize}
\item[(i)] There is a natural exact sequence:
$$0 \rightarrow \ch^2(k,C) \rightarrow \bigoplus_{v\in X^{(1)}} \mathbb{Q}_{\ell}/\mathbb{Z}_{\ell} \rightarrow \mathbb{Q}_{\ell}/\mathbb{Z}_{\ell} \rightarrow 0.$$
\item[(ii)] There is a natural exact sequence:
$$0 \rightarrow F \rightarrow \ch^2(k,A) \rightarrow \ch^2(k,B) \rightarrow \ch^2(k,C) \rightarrow 0$$
for some group $F$ which is killed by $m(\underline{\Gamma})$.
\item[(iii)] Let $S$ be a finite set of places de $k$. Assume that $\ell$ does not divide $m(\underline{\Gamma})$. Then there are natural isomorphisms:
\begin{gather*}
\ch^1_S(k,A) \cong \ch^1_S(k,B),\\
\Sha^2(k,A) \cong \Sha^2(k,B).
\end{gather*}
\end{itemize}

\end{lemma}

\begin{proof}[Proof]
\begin{itemize}
\item[(i)] We have an exact sequence:
$$0\rightarrow C \rightarrow  \bigoplus_{\underline{v}\in \underline{X}^{(1)}} \mathbb{Q}_{\ell}/\mathbb{Z}_{\ell}(1) \rightarrow \mathbb{Q}_{\ell}/\mathbb{Z}_{\ell}(1) \rightarrow 0.$$
Since for any finite field extension $E'/E$ the corestriction morphism $H^{1}(E',\mathbb{Q}_{\ell}/\mathbb{Z}_{\ell}(1)) \rightarrow H^{1}(E,\mathbb{Q}_{\ell}/\mathbb{Z}_{\ell}(1))$ is surjective, we get a commutative diagram with exact lines:
\begin{equation*}
\xymatrix{
0\ar[r] &H^2(k,C) \ar[r]\ar[d]^{\alpha_1} & \bigoplus_{v\in X^{(1)}} \text{Br}(k(v))\{\ell\} \ar[r]\ar[d]^{\alpha_2} & \text{Br}(k)\{\ell\} \ar[d]^{\alpha_3}  \\
0\ar[r] &\bigoplus_{\pi \in \Omega_k} H^2(k_{\pi},C) \ar[r] & \bigoplus_{v\in X^{(1)}} \bigoplus_{\pi \in \Omega_{k(v)}} \text{Br}(k(v)_{\pi})\{\ell\} \ar[r] & \bigoplus_{\pi \in \Omega_k} \text{Br}(k_{\pi})\{\ell\} 
}
\end{equation*}
\begin{equation*}
\xymatrix{
 \ar[r]& H^3(k,C) \ar[r]\ar[d]^{\alpha_4} & \bigoplus_{v\in X^{(1)}} H^3(k(v),\mathbb{Q}_{\ell}/\mathbb{Z}_{\ell}(1))\ar[d]^{\alpha_5}\\
 \ar[r] & \bigoplus_{\pi \in \Omega_k} H^3(k_{\pi},C) \ar[r] & \bigoplus_{v\in X^{(1)}} \bigoplus_{\pi \in \Omega_{k(v)}} H^3(k(v),\mathbb{Q}_{\ell}/\mathbb{Z}_{\ell}(1)).
}
\end{equation*}
In this diagram, $\alpha_4$ and $\alpha_5$ are isomorphisms by Poitou-Tate theorem. Moreover, $\alpha_3$ is injective by the Brauer-Hasse-Noether exact sequence. Hence we get an exact sequence:
$$0 \rightarrow \text{Coker}(\alpha_1) \rightarrow \text{Coker}(\alpha_2) \rightarrow  \text{Coker}(\alpha_3) \rightarrow  0.$$
But $\text{Coker}(\alpha_2) \cong \bigoplus_{v\in X^{(1)}} \mathbb{Q}_{\ell}/\mathbb{Z}_{\ell}$ and $\text{Coker}(\alpha_3) \cong \mathbb{Q}_{\ell}/\mathbb{Z}_{\ell}$ according to the Brauer-Hasse-Noether exact sequence. We have there fore obtained the required exact sequence:
$$0 \rightarrow \ch^2(k,C) \rightarrow \bigoplus_{v\in X^{(1)}} \mathbb{Q}_{\ell}/\mathbb{Z}_{\ell} \rightarrow \mathbb{Q}_{\ell}/\mathbb{Z}_{\ell} \rightarrow 0.$$
 \item[(ii)] By using exact sequence (\ref{spl2}), we get a commutative diagram with exact lines:
 \small
\begin{equation}\label{di1}
\xymatrix{
H^2(k,A) \ar[r]\ar[d]^{\beta_1} & H^2(k,B) \ar[r]\ar[d]^{\beta_2} & H^2(k,C) \ar[r]\ar[d]^{\beta_3} & H^3(k,A) \ar[r]\ar[d]^{\beta_4} & H^3(k,B) \ar[d]^{\beta_5} \\
\bigoplus\limits_{\pi\in \Omega_k} H^2(k_{\pi},A) \ar[r] & \bigoplus\limits_{\pi\in \Omega_k} H^2(k_{\pi},B) \ar[r] & \bigoplus\limits_{\pi\in \Omega_k} H^2(k_{\pi},C) \ar[r] & \bigoplus\limits_{\pi\in \Omega_k} H^3(k_{\pi},A) \ar[r] & \bigoplus\limits_{\pi\in \Omega_k} H^3(k_{\pi},B).  \\
}
\end{equation}
\normalsize
Moreover, by using the existence of the morphism $\psi$ of proposition \ref{gr}, we also have a commutative diagram:
\begin{equation}\label{di2}
\xymatrix{
H^2(k,A) \ar[rd]^{\cdot m(\underline{\Gamma})}\ar[rr]\ar[ddd]^{\beta_1} & & H^2(k,B) \ar[ld]^{\psi^*}\ar[ddd]^{\beta_2} \\
& H^2(k,A)\ar[d] & \\
& \bigoplus\limits_{\pi\in \Omega_k} H^2(k_{\pi},A) & \\
\bigoplus\limits_{\pi\in \Omega_k} H^2(k_{\pi},A) \ar[rr]\ar[ur]^{\cdot m(\underline{\Gamma})} & & \bigoplus\limits_{\pi\in \Omega_k} H^2(k_{\pi},B)\ar[ul]^{\psi^*}   \\
}
\end{equation}
Since $\beta_4$ and $\beta_5$ are isomorphisms by Poitou-Tate theorem and since $\beta_3$ is injective by theorem 2.5 of \cite{jannsen}, we deduce from diagrams (\ref{di1}) and (\ref{di2}) that there is an exact sequence:
\begin{equation}
0\rightarrow F \rightarrow \text{Coker}(\beta_1) \rightarrow \text{Coker}(\beta_2) \rightarrow \text{Coker}(\beta_3) \rightarrow 0
\end{equation}
in which $F$ is a group killed by $m(\underline{\Gamma})$. This finishes the proof.
\item[(iii)] Since $\ell$ does not divide $m(\underline{\Gamma})$, proposition \ref{gr} implies that $B \cong A \oplus C$. Hence $\ch^1_S(k,B)\cong \ch^1_S(k,A) \oplus \ch^1_S(k,C)$ and $\Sha^2(k,B) \cong\Sha^2(k,A)\oplus\Sha^2(k,C)$. But $\ch^1_S(k,C)=\Sha^2(k,C)=0$ by theorem 2.5 of \cite{jannsen}. Hence $\ch^1_S(k,B)\cong \ch^1_S(k,A)$ and $\Sha^2(k,B) \cong\Sha^2(k,A)$.
\end{itemize}

\end{proof}

\begin{lemma}\label{lemfin2}
Assume that $k$ is a number field and let $M$ be a Galois module over $k$ which is free of finite type as an abelian group. Set $\tilde{M} = M\otimes \mathbb{Q}_{\ell}/\mathbb{Z}_{\ell}(1)$. The group $\ch^2(k,\tilde{M})$ is isomorphic to $(\mathbb{Q}_{\ell}/\mathbb{Z}_{\ell})^{\rho}$ where $\rho = \text{rank}\left(M^{\text{Gal}(k^s/k)}\right)$.
\end{lemma}

\begin{proof}[Proof]
First observe that, for any finite place $\pi$ of $k$, the group $H^2(k_{\pi},\tilde{M})$ is divisible. Moreover, if $\Omega^{\infty}_k$ stands for the set of infinite places of $k$, Poitou-Tate's exact sequence implies that the restriction morphism:
$$H^2(k,\tilde{M}) \rightarrow \bigoplus_{\pi \in \Omega^{\infty}_k} H^2(k_{\pi},\tilde{M})$$
is surjective. Hence $\ch^2(k,\tilde{M})$ is divisible.\\

Let's now prove the lemma. If $M$ is a permutation module, the result follows immediately from the Brauer-Hasse-Noether exact sequence. In general, by Ono's lemma, there are a positive integer $m$ and an exact sequence of Galois modules:
$$0 \rightarrow M \times R_0 \rightarrow R_1 \rightarrow F \rightarrow 0$$
in which $R_0$ and $R_1$ are permutation modules and $F$ is finite. Hence one gets an exact sequence:
$$0 \rightarrow F(1) \rightarrow \tilde{M}^m \times \tilde{R}_0 \rightarrow \tilde{R}_1 \rightarrow 0.$$
We therefore obtain a commutative diagram with exact lines:
\small
\begin{equation}
\xymatrix{
H^2(k,F(1)) \ar[d]\ar[r] & H^2(k,\tilde{M})^m \times H^2(k,\tilde{R}_0) \ar[d]\ar[r] & H^2(k,\tilde{R}_1) \ar[d]\ar[r] & H^3(k,F(1))\ar[d]\\
\bigoplus\limits_{\pi \in \Omega_k} H^2(k_{\pi},F(1)) \ar[r] & \bigoplus\limits_{\pi \in \Omega_k}  H^2(k_{\pi},\tilde{M})^m \times H^2(k_{\pi},\tilde{R}_0) \ar[r] & \bigoplus\limits_{\pi \in \Omega_k} H^2(k_{\pi},\tilde{R}_1) \ar[r] & \bigoplus\limits_{\pi \in \Omega_k} H^3(k_{\pi},F(1))\\
}
\end{equation}
\normalsize
Hence there is an exact sequence:
$$0 \rightarrow F_1 \rightarrow \ch^2(k,\tilde{M})^m\times \ch^2(k,\tilde{R}_0) \rightarrow \ch^2(k,\tilde{R}_1) \rightarrow F_2 \rightarrow 0$$
in which $F_1$ and $F_2$ have finite exponent. Since the lemma holds for permutation modules and since the group $\ch^2(k,\tilde{M})$ is divisible, we immediately deduce that the lemma holds for $M$.
\end{proof}

\begin{proof}[Proof of theorem \ref{bhncdnbis}]
 Let  $\mathcal{R}$ be a complete set of representatives of $\underline{V}_1/\text{Gal}(k^s/k)$. By lemma \ref{qsd},  there is a finite subset $S$ of $\Omega_{k}$ such that the group $$\bigoplus_{\underline{v} \in \mathcal{R}} H^2(k^h_{\pi}\otimes F_{\underline{v}},J_{\underline{v}}(k^s)\{\ell\}(1))$$ vanishes for all $\pi \not\in S$. Hence, by proposition \ref{gr}, we get a commutative diagram with exact lines:
\small
\begin{equation*}
\xymatrix{
 H^1(k,\Xi(1)) \ar[r]\ar[d]^{\alpha_1} & \bigoplus_{\underline{v} \in \mathcal{R}} H^2(F_{\underline{v}},J_{\underline{v}}(k^s)\{\ell\}(1)) \ar[r]\ar[d]^{\alpha_2} & H^2(k,\br(\underline{K})(2)) \ar[d]^{\alpha_3} \\
\bigoplus_{\pi\in S} H^1(k^h_{\pi},\Xi(1)) \ar[r] & \bigoplus_{\pi\in \Omega_{k}} \bigoplus_{\underline{v} \in \mathcal{R}} H^2(k^h_{\pi}\otimes F_{\underline{v}},J_{\underline{v}}(k^s)\{\ell\}(1)) \ar[r] & \bigoplus_{\pi\in \Omega_{k}}H^2(k^h_{\pi},\br(\underline{K})(2)) 
}
\end{equation*}
\begin{equation*}
\xymatrix{
 \ar[r] & H^2(k,\Xi(1)) \ar[r]\ar[d]^{\alpha_4} & \bigoplus_{\underline{v} \in \mathcal{R}} H^3(F_{\underline{v}},J_{\underline{v}}(k^s)\{\ell\}(1))\ar[d]^{\alpha_5}\ar[r] & H^3(k,\br(\underline{K})(2)) \ar[d]^{\alpha_6}\\
  \ar[r] & \bigoplus_{\pi\in \Omega_{k}}H^2(k^h_{\pi},\Xi(1)) \ar[r] & \bigoplus_{\pi\in \Omega_{k}}\bigoplus_{\underline{v} \in \mathcal{R}} H^3(k^h_{\pi}\otimes F_{\underline{v}},J_{\underline{v}}(k^s)\{\ell\}(1))\ar[r]& \bigoplus_{\pi\in \Omega_{k}}H^3(k^h_{\pi},\br(\underline{K})(2)).
}
\end{equation*}

\normalsize

In this diagram, $\alpha_5$ and $\alpha_6$ are isomorphisms by Poitou-Tate theorem and $\alpha_2$ is an isomorphism by theorem 1.5(a) of \cite{jannsen}, since each $J_{\underline{v}}$ has good reduction at almost each place of $F_{ \underline{v}}$. A diagram chase shows that there is a natural exact sequence:
\begin{equation}\label{kernel}
\text{Coker}(\alpha_1) \rightarrow \text{Ker}(\alpha_3) \rightarrow \text{Ker}(\alpha_4) \rightarrow 0
\end{equation} 
and an isomorphism:
\begin{equation} \label{isomor}
\text{Coker}(\alpha_3) \cong \text{Coker}(\alpha_4).
\end{equation}
Lemmas \ref{lemfin} and \ref{lemfin2}, exact sequence (\ref{kernel}), isomorphism (\ref{isomor}) and proposition 1.2 of \cite{jannsen} imply the existence of a natural complex 
$\mathcal{C}_K$:
\small
\begin{equation}
\xymatrix@R=2mm@C=5mm{
& \text{(degree 1)} & \text{(degree 2)} & \text{(degree 3)} & \text{(degree 4)} & \\
0 \ar[r] &  H^4(K,\mathbb{Q}_{\ell}/\mathbb{Z}_{\ell}(3)) \ar[r] & \bigoplus\limits_{\pi\in \Omega_k} H^4(K_{\pi}^h,\mathbb{Q}_{\ell}/\mathbb{Z}_{\ell}(3)) \ar[r] & \bigoplus\limits_{v\in X^{(1)}} \mathbb{Q}_{\ell}/\mathbb{Z}_{\ell} \ar[r] & \mathbb{Q}_{\ell}/\mathbb{Z}_{\ell} \ar[r] & 0
}
\end{equation}
\normalsize
such that:
\begin{itemize}
\item[$\bullet$] $H^3(\mathcal{C}_K) = H^4(\mathcal{C}_K)=0$,
\item[$\bullet$] the group $H^2(\mathcal{C}_K)$ is the quotient of $\ch^2(k,A)$ by a subgroup killed by $m(\underline{\Gamma})$ and it is non-canonically isomorphic to $(\mathbb{Q}_{\ell}/\mathbb{Z}_{\ell})^{\rho}$ where $\rho = \text{rank}\left(H_1(\underline{\Gamma},\mathbb{Z})^{\text{Gal}(k^s/k)}\right)$,
\item[$\bullet$] if $\ell$ does not divide $m(\underline{\Gamma})$, then there exist a finite subset $S$ of $\Omega_k$ and a natural exact sequence:
\begin{equation} \label{hom1} 
\ch^1_S(k,A) \rightarrow H^1(\mathcal{C}_K) \rightarrow \Sha^2(k,A) \rightarrow 0.
\end{equation}
\end{itemize}

It only remains to prove that $H^1(\mathcal{C}_K)$ has finite exponent. Let $k'$ be a finite extension of $k$ such that $\text{Gal}(k'/k)$ acts trivially on $\underline{\Gamma}$. Then we have an exact sequence:
$$\ch^1_S(k',A) \rightarrow H^1(\mathcal{C}_{K \cdot k'}) \rightarrow \Sha^2(k',A) \rightarrow 0$$
and the groups $\ch^1_S(k',A)$ and $\Sha^2(k',A)$ are trivial. We deduce that $H^1(\mathcal{C}_{K \cdot k'})=0$, and a restriction-corestriction argument shows that  $H^1(\mathcal{C}_K)$ has finite exponent.
\end{proof}

\textit{Acknowledgements.} The author would like to specially thank David Harari and Jean-Louis Colliot-Thélène for interesting discussions about this article.


\begin{thebibliography}{99}

\bibitem{sga3}
M. Demazure, A. Grothendieck~: {\it S\'eminaire de G\'eom\'etrie Alg\'ebrique du Bois Marie 1962--64.}              A seminar directed by M. Demazure and A. Grothendieck with the
              collaboration of M. Artin, J.-E. Bertin, P. Gabriel, M.
              Raynaud and J-P. Serre,
              Revised and annotated edition of the 1970 French original, volume~7 of Documents mathématiques, Soci\'et\'e Math\'ematique de France, Paris (2011).

\bibitem{faltings} G. Faltings~: {\it $p$-adic Hodge theory}, J. Amer. Math. Soc., {\bf 1}, 255--299 (1988).

\bibitem{diego2} D. Izquierdo~:  {\it Dualité et principe local-global pour les anneaux locaux henséliens de dimension 2}, with an appendix by J. Riou, to appear in Algebraic Geometry.

\bibitem{jannsen2} U. Jannsen~: {\it Principe de {H}asse cohomologique}, in S\'eminaire de {T}h\'eorie des {N}ombres, {P}aris, 1989--90, Progr. Math., {\bf 102}, 121--140, Birkh\"auser Boston, Boston, MA (1992).

\bibitem{jannsen} U. Jannsen~: {\it Hasse principles for higher-dimensional fields}, Ann. Math. (2), {\bf 183} (1), 1--71 (2016).

\bibitem{kahn} B. Kahn~: {\it Resultats de "Pureté" pour les Variétés Lisses Sur un Corps Fini}, in Algebraic {$K$}-theory and algebraic topology ({L}ake {L}ouise, {AB}, 1991), NATO Adv. Sci. Inst. Ser. C Math. Phys. Sci., {\bf 407}, 57--62 (1993).

\bibitem{lipman} J. Lipman~: {\it Desingularization of two-dimensional schemes}, Ann. Math. (2), {\bf 107} (1), 151--207 (1978).

\bibitem{liu} Q. Liu~: {\it Algebraic geometry and arithmetic curves}, volume 6 of Oxford Graduate Texts in Mathematics, Oxford University Press, Oxford (2002).

\bibitem{milne} J. S. Milne~: {\it Arithmetic duality theorems}, second edition, BookSurge, LLC, Charleston, SC (2006).

\bibitem{pirutka} A. Pirutka~: {\it Invariants birationnels dans la suite spectrale de Bloch-Ogus}, J. K-Theory, {\bf 10}, 565--582 (2012).

\bibitem{raynaud} M. Raynaud~: {\it 1-motifs et monodromie géométrique}, in Périodes $p$-adiques (Bures-sur-Yvette, 1988), Astérisque, {\bf 223}, 295--320 (1994).

\bibitem{saito} S. Saito~: {\it Class field theory for two-dimensional local rings}, in Galois representations and arithmetic algebraic geometry ({K}yoto, 1985/{T}okyo, 1986), Adv. Stud. Pure Math., {\bf 12}, 343--373 (1987).

\bibitem{saitosato} S. Saito, K. Sato~: {\it A finiteness theorem for zero-cycles over {$p$}-adic fields}, with an appendix by U. Jannsen, Annals of Math., {\bf 172} (3), 1593--1639 (2010).

\bibitem{serre}
J-P. Serre~: {\it Cohomologie galoisienne}, volume~5 of Lecture Notes in Mathematics, fifth edition, Springer-Verlag, Berlin (1994).



\end{thebibliography}
\end{document}